\documentclass[final]{siamltex}

\pdfoutput=1

\usepackage{amsmath}
\usepackage{amssymb}
\usepackage{graphicx}
\usepackage{color}
\usepackage[english]{babel}
\usepackage{nicefrac}
\DeclareMathAlphabet{\mathbf}{OT1}{cmr}{bx}{it}

\newtheorem{example}{Example}[section]
\usepackage{cite}

\def\be#1\ee{\begin{equation}#1\end{equation}}
\newcommand{\bea}{\begin{eqnarray}}
\newcommand{\eea}{\end{eqnarray}}
\newcommand{\beas}{\begin{eqnarray*}}
\newcommand{\eeas}{\end{eqnarray*}}
\newcommand{\bfR}{\mathbb{R}}
\newcommand{\bfC}{\mathbb{C}}

\newcommand{\sn}{\mathop{\rm sn}\nolimits}
\newcommand{\dn}{\mathop{\rm dn}\nolimits}
\newcommand{\capl}{\mathop{\rm cap}\nolimits}
\newcommand{\Kei}{\mathrm{K}}
\newcommand{\srd}{{\sqrt{\delta}}}

\newcommand{\tK}{{\widetilde K}}
\newcommand{\vnull}{\boldsymbol{0}}
\newcommand{\s}{\sigma}

\newcommand{\e}{\mathrm{e}}
\newcommand{\SG}[1]{{\color{black} #1}}

\newcommand{\rev}[1]{{\color{black}{#1}}}

\title{Near-optimal perfectly matched layers\\ for indefinite  Helmholtz problems}

\author{Vladimir Druskin\thanks{Schlumberger-Doll Research, 1 Hampshire
  St., Cambridge, Massachusetts, 19104-2688
  ({\tt druskin1@boston.oilfield.slb.com}).}
  \and Stefan G\"uttel\thanks{School of Mathematics, The University of Manchester,
Alan Turing Building,
Manchester, M13\,9PL, United Kingdom ({\tt stefan.guettel@manchester.ac.uk}).}
  \and Leonid Knizhnerman\thanks{Mathematical Modelling Department,
  Central Geophysical Expedition, Narodnogo Opolcheniya St., house~38,
  building~3, 123298, Moscow, Russia ({\tt lknizhnerman@gmail.com}).}}

\begin{document}

\maketitle

\begin{abstract}
 A new construction of an absorbing boundary condition for indefinite Helmholtz problems on  unbounded domains  is presented.
 This construction is based on a near-best uniform rational interpolant of the inverse square root function
 on the union of a negative and positive real interval,  designed with the help of a classical result by Zolotarev. Using Krein's interpretation of a Stieltjes continued fraction, this interpolant can be converted into a three-term finite difference discretization of a perfectly matched layer (PML)  which  converges exponentially
 fast in the number of grid points. The convergence rate is asymptotically optimal  for both propagative and evanescent wave modes.
 Several numerical experiments and illustrations are included.
\end{abstract}



\pagestyle{myheadings}
\thispagestyle{plain}
\markboth{V.~DRUSKIN, S.~G\"UTTEL, AND L.~KNIZHNERMAN}%
{PERFECTLY MATCHED LAYERS FOR INDEFINITE HELMHOLTZ PROBLEMS}

\reversemarginpar

\section{Introduction}

An important task in science and engineering is the numerical solution
of a partial differential equation (PDE) on an unbounded spatial domain. 
\rev{Unbounded spatial domains need to be truncated for computational purposes and this turns out to be particularly difficult when the PDE models wave-like phenomena. In this case the  solution may not decay rapidly towards the truncation boundary and artificial reflections and resonances may pollute the numerical solution. The prototype of such notorious PDEs is the Helmholtz equation, which models the propagation of electromagnetic or acoustic fields from a source with a single frequency $k>0$   
\begin{equation}\label{eq:helmh}
c^2\Delta u + k^2 u = 0.
\end{equation} 
\rev{As a motivating example from geophysics, this equation may be posed 
on an unbounded half-space corresponding to the Earth's subsurface and the variable wave speed $c$ may be caused by variations in the sedimental composition, see Figure~\ref{fig:drawing}. 
In  seismic exploration a pressure wave signal of 
frequency $k$ is emitted by an acoustic transmitter  placed on the Earth's surface or below, travels through the underground, and is then logged by receivers. From these measurements geophysicists try to infer variations in the wave speed $c$ which then allows them to draw conclusions about the  subsurface composition.} Clearly, the spatial domain for this problem needs to be truncated and there are various ways for achieving this,
with a very popular approach} being known as
\emph{perfectly matched layer} (PML, see \cite{Balslev&Combes, Ber94, ChewWeedon}).

A perfectly matched layer can be seen as a localized modification of the
spatial discretization scheme to absorb the \rev{waves exiting the computational domain}.
In a finite difference framework such layers typically lead to variable complex-valued
step sizes, which is why this approach is sometimes also referred to as
\emph{complex coordinate stretching}. The aim of an efficient PML is to
achieve a strong absorption effect by adding only a few number of layers.
The aim of this work is to extend a modern finite-difference construction
of perfectly matched layers which are near-optimal for indefinite Helmholtz problems,
that is, they achieve near-best possible absorption for a given number of layers.
The number of required layers is critical in particular for large-scale simulations   of three-dimensional exterior  problems.
A variety of such problems arise, for example,  in oil and gas exploration, and 
near-optimal grids are part of almost all electromagnetic simulators used at Schlumberger \cite{davydycheva_etal2002,Abubakar_etal,Zaslavsky_etal}. \rev{
Other applications of effective discretizations of exterior domains include homogenization theory, photonic crystals, 
energy-driven pattern formation,  and the modelling of 
biologic cell communication (see, e.g.,  \cite{Muratov1,Muratov2,Muratov3}).
%
}

\begin{figure}[t]
\begin{center}
\includegraphics[width=8cm]{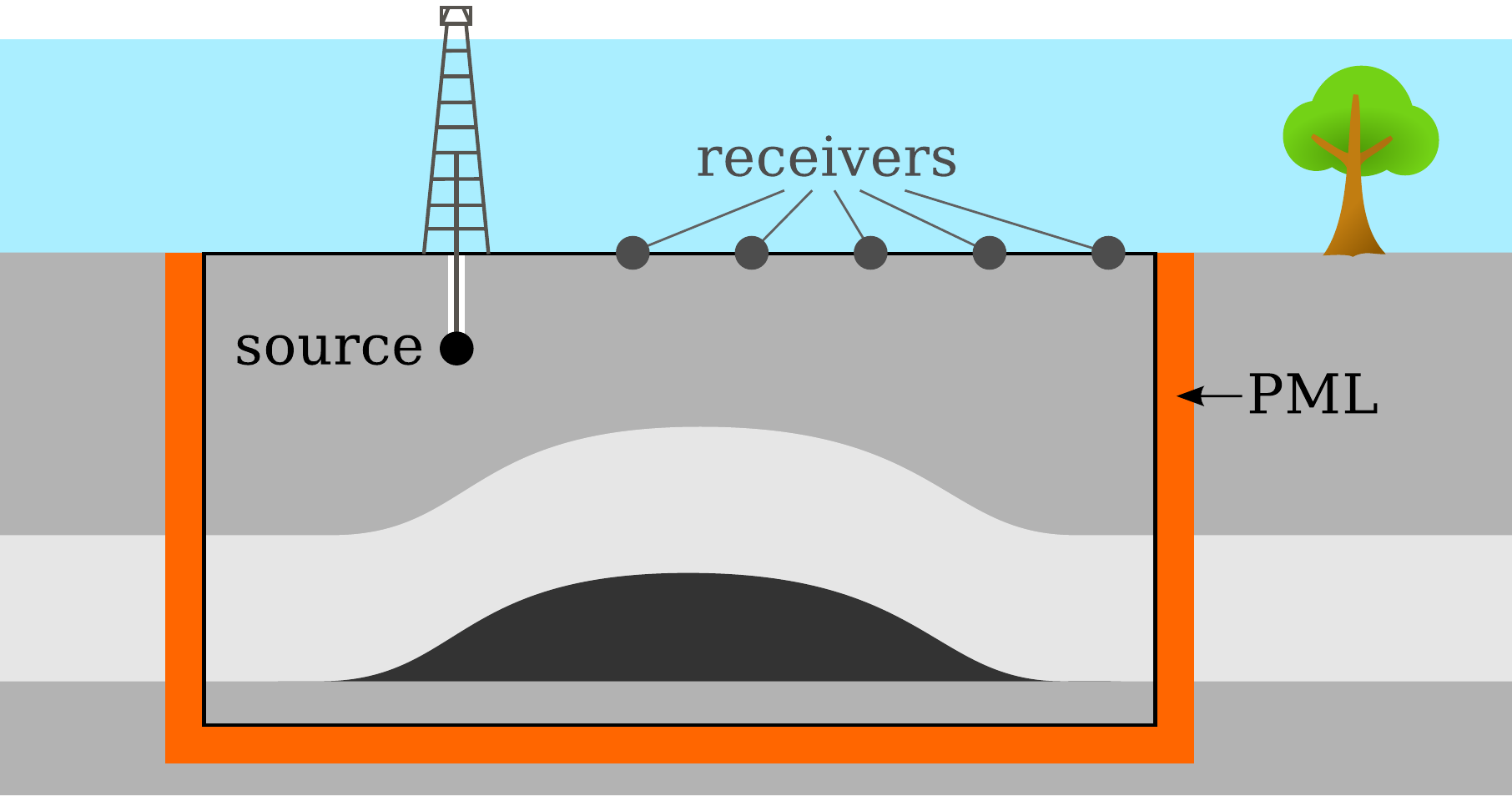}
\end{center}
\caption{A typical setup in {seismic} geophysical exploration where a
 source emits {pressure waves} into the Earth's subsurface which are then logged
at (multiple) receivers. The wave propagation is modelled by the Helmholtz equation \rev{\eqref{eq:helmh}. The varying shades of gray in the Earth's underground indicate different values of the wave speed $c=c(x,y)$. The perfectly matched layers developed in this paper allow for a variation in $c$ tangential to  the boundary of the computational domain. We will return to this example in section~\ref{sec:tensorpml}.}  
} \label{fig:drawing}
\end{figure}
%



\subsection{Outline of this work}
\SG{We will now give a short overview of this work and explain the structure of the paper.
Let us start by considering a prototype of a differential equation on an unbounded domain,}
the  two-point boundary value problem
\begin{equation}\label{eq:bvp}
    \frac{\partial^2}{\partial x^2} \mathbf{u} = \mathbf{A} \mathbf{u}, \quad
    \frac{\partial}{\partial x} \mathbf{u} \big|_{x=0} = -\mathbf{b}, \quad \mathbf{u} \big|_{x=+\infty} = \vnull,
\end{equation}
where $\mathbf{A}\in\mathbb{C}^{N\times N}$ is nonsingular and $\{\mathbf{b},\mathbf{u}(x)\}\subset\mathbb{C}^N$.
If $\mathbf{A}$ is a discretization of a differential operator on some spatial domain $\Omega \subseteq \bfR^\ell$, then (\ref{eq:bvp}) is a
semidiscretization of an $(\ell+1)$-dimensional partial differential equation on $ [0,+\infty) \times \Omega$.
Assuming that problem (\ref{eq:bvp}) is well posed (which may require some additional conditions like, e.g., the limiting absorption principle discussed below), its exact solution can be given in terms of matrix functions as $\mathbf{u}(x) = \exp(-x \mathbf{A}^{1/2})\mathbf{A}^{-1/2}\mathbf{b}$. In particular, at $x=0$ the solution is given as
\begin{equation}\label{eq:invsqrtmv}
        \mathbf{u}(0)= F(\mathbf{A}) \mathbf{b}, \quad F(z) = z^{-1/2}.
\end{equation}
The function $F(z)$ is often referred to as the \emph{impedance function} (also known as \emph{Weyl function}), and it completely characterizes the reaction of the unbounded domain
to an external force \cite{KK}.
The relation \eqref{eq:invsqrtmv} allows
 for the exact conversion of the Neumann data $-\mathbf{b}$ at the boundary $x=0$ into the Dirichlet data $\mathbf{u}(0)$, without the need for solving \eqref{eq:bvp} on its unbounded domain. \SG{This is why $F(\mathbf{A})$ is often referred to as the \emph{Neumann-to-Dirichlet (NtD) operator}.}

When solving wave scattering problems one typically \rev{deals with a
discretization matrix $\mathbf{A}$ of the negative shifted Laplacian $-c^2\Delta - k^2$ on  $\Omega \subset \bfR^\ell$. 
Under the assumption that $c$ does not depend on $x$, problem (\ref{eq:bvp}) is a semidiscretization of the indefinite Helmholtz equation \eqref{eq:helmh} on the domain $[0,+\infty)\times \Omega$.
In this case the matrix $\mathbf{A}$ is (similar to a matrix) of the form} 
\be\label{real}\mathbf{A}=\mathbf{L}-k^2\mathbf{I},\ee
where $\mathbf{L}\in\mathbb{C}^{N\times N}$ is Hermitian positive definite,
 $\mathbf{I}\in\mathbb{R}^{N\times N}$ is the identity matrix, and $k^2>0$ is not in the spectrum of $\mathbf{L}$.
For a solution of  (\ref{eq:bvp}) to be unique we impose the limiting absorption principle (see, e.g., \cite{Taylor}).
This means that for a real number $k$ we define $\mathbf{u}$ as a limit of solutions $\mathbf{u}^{(k + i\epsilon)}$ of \eqref{eq:bvp} with wave numbers
$k + i\epsilon$ ($\epsilon>0$) instead of $k$, i.e.,
\begin{equation}\label{eq:lap}
\mathbf{u} = \lim_{\epsilon \searrow 0} \mathbf{u}^{(k + i\epsilon)}.
\end{equation}
This uniquely defines
the value $F(\mathbf{A})=\mathbf{A}^{-1/2}$, notwithstanding that some eigenvalues
of $\mathbf{A}$ may lie on the standard branch cut \SG{of $F(z)$}.

We will now  outline our construction in the following sections, which combines ideas of the eminent mathematicians Y.\,I.~Zolotarev (1847--1878), T.\,J.~Stieltjes (1856--1894), and M.\,G.~Krein (1907--1989). 
\SG{The main aim in section~\ref{sec:contruction} is to approximate $F(z)$ by a rational interpolant $R_n(z)$ of type $(n-1,n)$, so that $R_n(\mathbf{A})$ can be seen as an approximate NtD operator, mapping the Neumann data $-\mathbf{b}$ to the  Dirichlet data $R_n(\mathbf{A}) \mathbf{b}$.
 Clearly, the \rev{weighted} 2-norm approximation error of this map is
\[\|R_n(\mathbf{A}) \mathbf{b}-F(\mathbf{A})\mathbf{b}\| = \sqrt{\sum_{j=1}^N | R_n(\lambda_j-k^2)-F(\lambda_j-k^2) |^2 |b_j|^2},\]
where $b_j= \mathbf{v}_j^*\mathbf{b}$ and $(\lambda_j,\mathbf{v}_j)$ are the eigenpairs of $\mathbf{L}$ with $\|\mathbf{v}_j\|=1$. 
We have $\lambda_1<k^2<\lambda_N$ and thus arrive at the problem of scalar rational approximation of $F(z)$ on
the union of a positive and a negative real interval.
Our rational interpolant $R_n(z)$ is obtained
by combining two optimal Zolotarev interpolants constructed for the two intervals separately.
 For illustration purposes we have graphed the relative error of such a function in Figure~\ref{fig:example1}.
In addition to the \rev{explicit} construction of such approximants, section~\ref{sec:contruction} also contains a novel detailed  convergence analysis, with the more technical proofs given in the appendix.}

In section \ref{sec:grids} we will show that the rational function $R_n(z)$ can be converted into an equivalent three-term finite difference scheme on a nonuniform grid with $n$ points. This is achieved by formally rewriting $R_n(z)$ as a Stieltjes continued fraction  and using Krein's interpretation of that fraction as a finite-difference scheme. However, due to the non-Stieltjes nature  of $R_n(z)$  (its poles may lie on a curve in complex plane, as shown in Figure~\ref{fig:example1e})  the continued fraction coefficients  can also be complex, which results in a finite difference scheme with complex-valued grid steps.
This scheme allows for the simple and  efficient computation of an NtD map and the construction of an absorbing boundary layer for indefinite Helmholtz problems.
The  
near-optimality of $R_n(z)$ implies that
the number of required grid points is close to smallest possible.
A summary of an algorithm for computing this grid is given in section~\ref{sec:algsumm}.



Section~\ref{sec:impl} is devoted to the adaptation of our PML construction to a second-order finite difference framework. In particular, in section~\ref{sec:discr}, we extend  {our optimal} rational approximation approach to the infinite lattice problem. \rev{Our analysis carries over to this problem, thereby providing a novel theoretical justification for the exponential error reduction in our perfectly matched layer as the number of grid points increases.}

Finally, in section~\ref{sec:numex} we demonstrate the high accuracy and exponential convergence of our
perfectly matched layer with several numerical examples.

\begin{figure}[t]
\begin{center}
\includegraphics[width=0.75\textwidth]{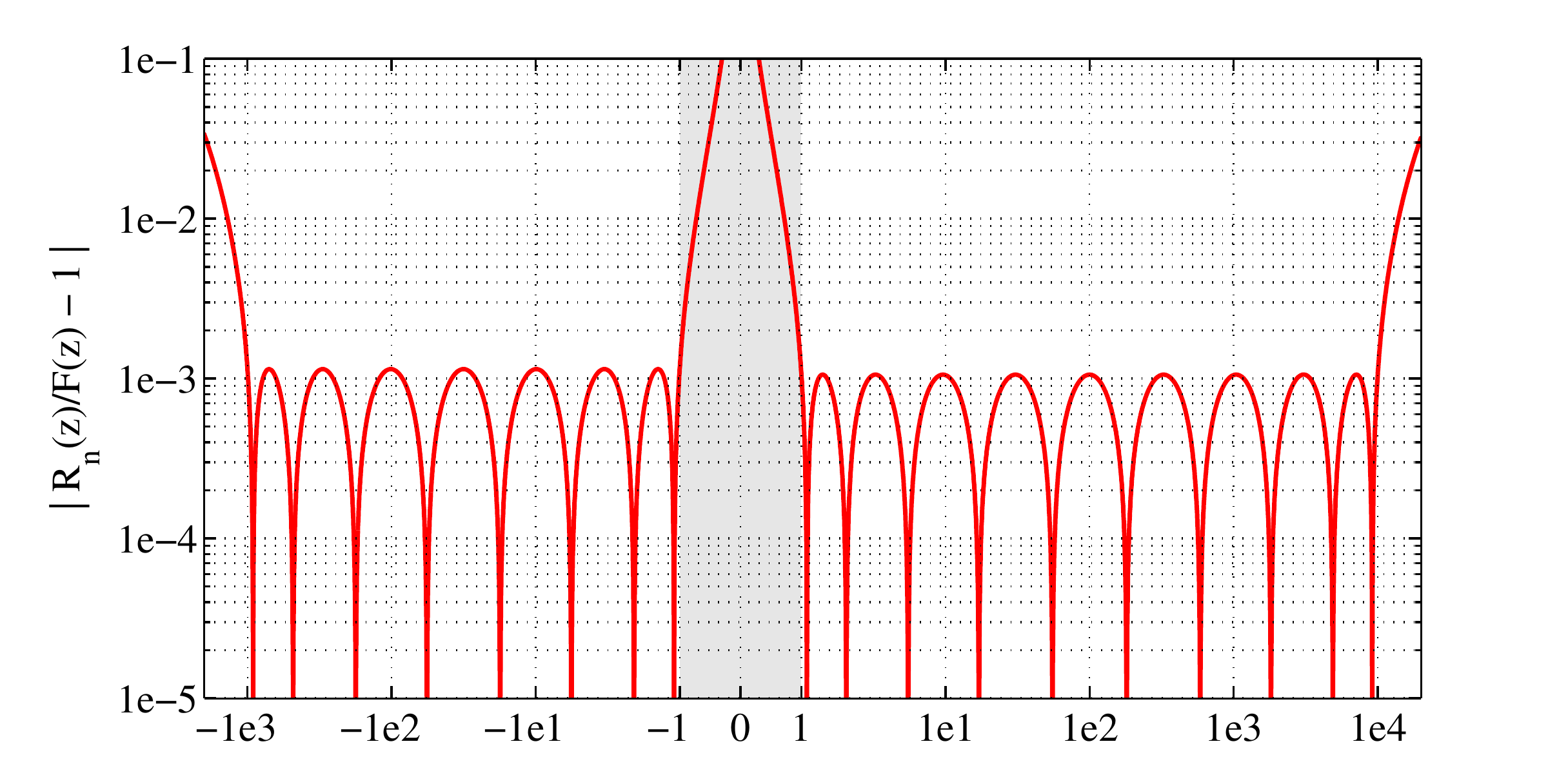}
\end{center}
\caption{Relative error $| R_n(z)/F(z) - 1 |$ of a rational approximant $R_n(z)$ for $F(z)=z^{-1/2}$ on $[-1\e3,-1]\cup [1,1\e4]$. We have adopted a special plotting type for simultaneously visualizing large intervals
on the negative and positive real semiaxes in logarithmic scales, with the gray linear region in the middle gluing the two intervals together. The rational function $R_n(z)$ is of type $(n-1,n)$, $n=9$, and
 it has been constructed by combining two Zolotarev interpolants with $m_1=8$ and $m_2=10$ interpolation nodes for the  negative and positive intervals, respectively. Visually the solution of our complex rational approximation problem  behaves similarly to the max-norm optimal errors of  the real  problems, i.e., it shows ``equal ripples'' on the targeted intervals  \emph{(}although the Chebyshev alternation theory
\emph{\cite[Ch.~II]{Akhiezer}} is not applicable in the complex case \emph{\cite{Varga}}\emph{)}.} \label{fig:example1}
\end{figure}

\subsection{Review of related work}

It was already shown in \cite{DK99,IDK00} that a rational approximant $R_n(z)$ of type
$(n-1,n)$ for the function $F(z)$ can be converted into an equivalent three-term finite difference scheme
on a special nonuniform grid with $n$ points, mapping the Neumann data $-\mathbf{b}$ to the  Dirichlet data $R_n(\mathbf{A}) \mathbf{b}$.
In these papers the authors were mainly concerned with a special instance of \eqref{eq:bvp} where $\mathbf{A}$ corresponds to a discretization of the negative Laplacian $-\Delta$,
in which case $\mathbf{A}$ is a real symmetric positive  definite matrix.
The error of the approximate Neumann-to-Dirichlet (NtD) map is then bounded by the maximum  of  $|R_n(z) - F(z)|$ on the
positive spectral interval of $\mathbf{A}$. Approximation theory allows for the construction of exponentially convergent rational functions $R_n(z)$ with a
convergence rate weakly dependent on  the condition number of $\mathbf{A}$,
thus producing  a  three-term finite difference scheme with a so-called \emph{optimal grid}   (also known as \emph{finite-difference Gaussian rule} or \emph{spectrally matched grid}). The connection of $R_n(z)$ and this grid is inspired by Krein's mechanical interpretation of a Stieltjes continued fraction \cite{KK}. It was shown in \cite{sharichka} that the same grids produce exponentially convergent NtD maps even for problems arising from the semidiscretization of
\emph{anisotropic} elliptic PDEs and systems with mixed second-order terms, i.e., when the second-order ODE system
in  (\ref{eq:bvp}) is modified by adding a first-order term.

{It should be noted that the positive and negative eigenmodes {of $\mathbf{A}$} correspond to so-called \emph{evanescent} and \emph{propagative solutions} $\eta(\mathbf{A})\mathbf{u}(x)$ and $\eta(-\mathbf{A})\mathbf{u}(x)$, respectively,
with $\eta(s)$ denoting the Heaviside step function.
The evanescent modes, i.e., the nonzero eigenmodes in the spectral decomposition of $\eta(\mathbf{A})\mathbf{u}(x)$, decay exponentially
as $x$ increases (hence the name). Therefore a simple, though possibly not the most efficient, way to absorb them is to truncate the domain at some (sometimes quite significant) distance from the targeted area of interest, and then to deal with the propagative modes alone. On the other hand, the norm $\|\eta(-\mathbf{A})\mathbf{u}(x) \|$ 
 does actually not depend on $x$, so simple boundary truncation will not be effective for absorbing propagative modes.}

In their  seminal paper \cite{EM79}, Enquist and Majda computed $R_n(z)$ as a Pad\'e approximant of $F(z)$ at some real negative point and then evaluated it via continued fraction-type recursions. This approach yielded exponential convergence on the negative real semiaxis, however, with the rate quickly deteriorating towards the origin.   Another celebrated approach for absorbing  propagative modes is called \emph{complex scaling} and was originally introduced in \cite{Balslev&Combes}
for molecular physics calculations. It is also known as \emph{perfectly matched layer} (PML), a term  coined in the influential  work \cite{Ber94}, where it was independently rediscovered and adapted for time-domain wave propagation. We will use the latter term because it seems to be more established in the wave propagation literature.  The well-posedness of the PML formulation was studied in \cite{Joly, Hagstrom}.
The essence of the PML approach is a complex coordinate transformation
which changes purely imaginary exponentials of propagative modes  to complex decaying ones, thus, in principle, allowing reflectionless domain truncation \cite{Balslev&Combes,ChewWeedon}. However, coarse PML discretizations introduce undesirable numerical reflections which   decay  rather slowly with the grid size in case of low-order discretization schemes.  This problem was partially circumvented in \cite{Asvadurov_etal} for the solution of time-domain wave problems, where the optimal gridding approach  was extended to PML discretizations. By choosing an appropriate \emph{purely imaginary} grid this approach allowed for the construction of all possible rational interpolants $R_n(z)$ for  $F(z)$  on a real negative interval, including the Pad\'e approximants
constructed in \cite{EM79}, and preferably the  best uniform  approximants targeting the spectral support of the expected  solution. See also \cite{Lisitsa} and \cite{DrRemis} for  adaptations of the optimal gridding approach to  the hyperbolic elasticity system  and the Helmholtz equation, respectively.  A non-optimal PML layer for absorbing both evanescent and propagative modes in dispersive wave equations has been proposed in \cite{ZG06}.
However, the problem of designing discrete PMLs which are optimal for both wave modes remained open.

 Our construction in section~\ref{sec:contruction} is inspired by a ``trick'' originally used by Zolotarev and Newmann,  writing the relative approximation error $R_n(z)/F(z) - 1$ in terms of
$H_m(s)/H_m(-s)$, where $H_m(s)$ is a polynomial of degree $m=2n$, $s^2 = z$.  This  trick  was rediscovered in \cite{GT00,GL06}, where $H_m(s)/H_m(-s)$ was identified with the numerical reflection coefficient,  and a continued-fraction absorbing condition was explicitly constructed in terms of the roots of $H_m(s)$ and introduced in the PDE discretization via a so-called trapezoid finite element method. However, these important papers fell short of introducing optimal approximants. In addition to the construction of these approximants,
 section~\ref{sec:contruction} also contains a novel detailed  convergence analysis. To make our paper more pleasant to read we have decided to present the technical proofs in an appendix.

In an unfinished report \cite{DGH}, the authors suggested to split $H_m(s)$ into the product of polynomials with real and imaginary roots, thus decoupling the approximation problems on the positive and negative intervals. It was then suggested to apply conventional optimal rational approximants on
each of the two intervals, and the resulting error was only determined by the largest error of  these two approximants. A drawback of such an approach is that it  requires the splitting of the PML grid into two subdomains with nonlocal finite difference stencils at the conjugation interfaces.

\section{Construction of a near-optimal approximant on two intervals}\label{sec:contruction}

The function $z^{-1/2}$ is commonly defined in the complex plane $\bfC$
with the slit $(-\infty,0]\subset\bfR$. However, in our application we
need an analytic continuation
$F(z)$ of $z^{-1/2}$ from the
positive real semiaxis $\bfR_+=\{x\in\bfR\mid x>0\}$ to $-\bfR_+$ in accordance with the limiting absorption principle (\ref{eq:lap}), i.e., attaining the values
\be \label{fnss}\nonumber
F(z)=-i(-z)^{-1/2} \qquad \text{for} \quad z\in -\bfR_+,
\ee
and the principal value of the square root for  $z\in \bfR_+$. We will therefore assume in the following that $F(z)$ is defined in $\bfC$ with the branch cut in the lower half-plane.

Following~\cite{DGH}, we now construct a rational interpolant $R_n(z)$ of type $(n-1,n)$ to $F(z)$ on the union
$K$ of two real intervals
\bea \label{apprset}\nonumber
K=K_1\cup K_2, \quad K_1=[a_1,b_1], \quad K_2=[a_2,b_2], \\
a_1<b_1<0<a_2<b_2, \nonumber
\eea
using solutions of a classical Zolotarev problem on each of the two intervals.
In view of the definition (\ref{real}), these intervals will correspond to the  spectral subintervals $[\lambda_1-k^2,\lambda_{i_0}-k^2]$   and $[\lambda_{{i_0}+1}-k^2,\lambda_N-k^2]$ (or their estimates), respectively,
where
$\lambda_1\leq \cdots\leq \lambda_{i_0}<k^2<\lambda_{i_0+1}\leq\cdots\leq \lambda_N$.

Separating the odd and even parts of a polynomial
$H_m$ of degree $m=2n$, we define polynomials $P_{n-1}$ and $Q_n$ of degrees
$\le n-1$ and $n$, respectively, such that
\be \label{defpq}
H_m(s)=-s \, P_{n-1}(s^2)+Q_n(s^2).
\ee
The rational function
\be \label{defR}
R_n(z) = \frac{P_{n-1}(z)}{Q_n(z)}
\ee
will be considered as an approximant for $F(z)$
on $K$. We have
\be \label{apprerr}
sR_n(s^2)=\frac{s\, P_{n-1}(s^2)}{Q_n(s^2)} = \frac{H_m(-s)-H_m(s)}{H_m(-s)+H_m(s)}\, ,
\ee
and thereby obtain an expression of the relative averaged approximation error as
\[2 \frac{\left|F(s^2)-R_n(s^2)\right|}{\left| F(s^2)+R_n(s^2)\right|}=2 \left|\frac{H_m(s)}{H_m(-s)}\right|.\]
Following \cite[Section~2]{DGH}, we can split the  approximation problem on $K$ into two independent problems on $K_1$ and $K_2$.

\begin{lemma}\label{split}
Let $m_1$ and $m_2$ be positive integers such that $m=m_1+m_2$, and let
$H_{m_1}$ and $H_{m_2}$ be polynomials of degrees $m_1$ and $m_2$  with roots on $F(K_1)$ and $F(K_2)$,  respectively. Define
\[
H_m(s)=H_{m_1}(s)H_{m_2}(s).
\]
Then \[
\max_{s\in F(K_1)} \left|\frac{H_m(s)}{H_m(-s)}\right|
= \max_{s\in F(K_1)}\left|\frac{H_{m_1}(s)}
{H_{m_1}(-s)}\right|\phantom{.}
\]
and
\[
\max_{s\in F(K_2)} \left|\frac{H_m(s)}{H_m(-s)}\right|
= \max_{s\in F(K_2)}\left|\frac{H_{m_2}(s)}
{H_{m_2}(-s)}\right|.
\]\end{lemma}
 \begin{proof}
This lemma immediately follows from the equalities
\[
\left|\frac{H_{m_1}(s)}
{H_{m_1}(-s)}\right|=1
\qquad \mbox{if} \quad
s\in F(K_2),
\]
and reciprocally
\[
\left|\frac{H_{m_2}(s)}
{H_{m_2}(-s)}\right|=1
\qquad \mbox{if} \quad
s\in F(K_1).
\]\end{proof}

 Let us  consider a single real interval $[c,d]$ with $0<c<d$, and
the problem of finding a real monic polynomial $Z^{(c,d)}_m$ of degree $m\geq 1$ (denoted as $Z^{(c,d)}_m\in \mathcal{P}_{m,\mathrm{real}}$) which attains the minimum in the Zolotarev problem
\be \label{inZolpr}
E^{(c,d)}_m=
\min_{Z\in \mathcal{P}_{m,\mathrm{real}}}
\ \max_{c\le s\le d} \left|\frac{Z(s)}{Z(-s)}\right|.
\ee
 It is known from \cite{Zol,ML05} that this minimizer $Z^{(c,d)}_m$ exists uniquely, that
 its roots $s^{(c,d)}_j$ ($j=1,\ldots,m$) are located in $(c,d)$, and that they are expressible in terms of elliptic integrals. More details are given in the appendix, in particular, formula (\ref{zolotar}).

We choose positive integers $m_1$ and $m_2$ and introduce the polynomial
\be \label{defh}
H_m(s)=Z^{(\sqrt{-b_1},\sqrt{-a_1})}_{m_1}(-is)\cdot
Z^{(\sqrt{a_2},\sqrt{b_2})}_{m_2}(s)
\ee
of degree $m=m_1+m_2$. From Lemma~\ref{split}
we obtain the following result.

\begin{proposition}\label{main0}
The polynomial $H_m(s)$ defined in \eqref{defh} satisfies
\[
\max_{s\in F(K)} \left|\frac{H_m(s)}{H_m(-s)}\right|=\max\left\{E^{(\sqrt{-b_1},\sqrt{-a_1})}_{m_1}\ , \ E^{(\sqrt{a_2},\sqrt{b_2})}_{m_2}\right\}.
\]
\end{proposition}

It is well known that the classical Zolotarev functions in \eqref{inZolpr} converge exponentially.
 Let us denote by
$\rho^{(\delta)}$ the Cauchy--Hadamard convergence rate of $Z^{(c,d)}_{m}$, i.e.,
\be\label{rhodelta}\nonumber
\rho^{(\delta)}=\lim_{m\to\infty}\sqrt[m] {E^{(c,d)}_m}, \quad \delta=c/d.\ee
An exact expression of $\rho^{(\delta)}$ in terms of elliptic integrals is given in (\ref{rho12}).
For  small \emph{interval ratios} $\delta$ one can derive a simple approximate expression
\[\rho^{(\delta)}  \approx \exp\left(-\frac{\pi^2}{4\log\frac{2}{\srd}}\right)
\]
in terms of elementary functions  \cite[Appendix~A]{IDK00}.
 This expression shows the weak dependence of the  Cauchy--Hadamard convergence rate on the interval ratio $\delta$.

In view of Proposition~\ref{main0}, $m_1$ and $m_2$ should be chosen to balance the errors of both Zolotarev functions. One way of achieving this is by setting
\be \label{eq:rho12}
\rho_1=\rho^{(\sqrt{b_1/a_1})}, \qquad \rho_2=\rho^{(\sqrt{a_2/b_2})}
\ee
and
\be \label{m12}
m_1= m\cdot \frac{\log\rho_2}{\log\rho_1+\log\rho_2}+\theta, \quad m_2=m-m_1,  \quad |\theta|\le 1/2,\ee
where $\theta$ is chosen to round  $(m \log\rho_2) / (\log\rho_1+\log\rho_2)$ to the nearest integer. We are now in the position to formulate a near-optimality result for the obtained approximant.

\smallskip

\begin{theorem} \label{T2}
Let us denote
\be \label{defrho}
\rho=\exp\left(\frac{\log\rho_1\log\rho_2}{\log\rho_1+\log\rho_2}
\right).
\ee
Let  the polynomial $H_m$ be defined by (\ref{defh}), the polynomials $P_{n-1}$
and $Q_n$ defined by (\ref{defpq}), the rational fraction $R_n$ defined by (\ref{defR}), and $m=2n$.
 Further let the conditions
(\ref{m12}) and
\be \label{T2cond}
2\max\left\{\rho_1^{-1/2},\rho_2^{-1/2}\right\}\rho^m<1
\ee
be satisfied. Then the upper relative error bound
\be \label{T2res1}
\max_{z\in K} \left|\frac{R_n(z)}{F(z)}-1\right|
\le \frac{4\max\left\{\rho_1^{-1/2},\rho_2^{-1/2}\right\}\rho^m}
{1-2\max\left\{\rho_1^{-1/2},\rho_2^{-1/2}\right\}\rho^m}
\ee
holds.
On the other hand, if $P$ and $Q\not\equiv 0$ are arbitrary polynomials of
degrees $\le n-1$ and $\le n$, respectively, then $R=P/Q$ satisfies the lower error bound
\be \label{T2res2}
\max_{z\in K} \left|\frac{R(z)}{F(z)}-1\right|
\ge \frac{2\rho^m}{1+\rho^m}.
\ee
\end{theorem}

This theorem, whose proof is given in the appendix, implies that the upper error bound for our \emph{Zolotarev approximant $R_n(z)$} and the lower bound for the best possible approximant have the same Cauchy--Hadamard convergence rate
$\rho$, i.e., our approximant is asymptotically optimal in the Cauchy--Hadamard sense.
  As  is also demonstrated by the following numerical example (and the corresponding  Table~\ref{tab:errors}), the Zolotarev approximant can be worse than the best possible approximant only by   a moderate factor. We should point out that, unlike their real counterparts,  complex max-norm optimal rational approximation problems  are generally not  convex and may have non-unique solutions \cite{Varga}. It therefore seems unlikely that the near-optimality result of Theorem~\ref{T2}  can be improved  significantly.


\begin{example}
Let us, as in Figure~\ref{fig:example1}, consider the problem of approximating $F(z) = z^{-1/2}$ by a rational function $R_n(z)$ of type $(n-1,n)$ on the union of two intervals $K = [a_1,b_1]\cup [a_2,b_2] = [-1\e3,-1]\cup [1,1\e4]$. Using the exact formula \eqref{rho12}  we  calculate
\[
    \rho_1 \approx  0.361, \quad \rho_2 \approx  0.439, \quad \rho \approx  0.634.
\]
In Table~\ref{tab:errors} we list the error bounds of Theorem~\ref{T2} for various values of $m = 2n$ together with the actual approximation error.
The calculations confirm the bounds and show that they are roughly of the same order, i.e., our approximants $R_n(z)$ have relative errors of the same order as the best possible approximants.

The logarithmic surface plot in Figure~\ref{fig:example1e} shows the relative error  $|R_n(z)/F(z) - 1|$ for the case $n=9$ (the same as in~Figure~\ref{fig:example1}). Note how
the poles align on a curve in the lower-left quadrant of the complex plane.
We speculate that this curve asymptotically (as $n\to\infty$)
approximates the shifted branch cut $C$ of the analytic continuation of
$F(z)$ into the lower half-plane, and that $C$ possesses the so-called S-property
(``symmetry property'', see \cite{Gonchar84, Stahl86, GR87}) with respect
to $K$. This would imply that
the equilibrium charge of the condenser $(K,C)$ has a logarithmic potential which is
(constant and) minimal on $K$ over all ``attainable'' branch cuts.
Our experiments also suggested that the curve $C$ coincides exactly with the negative imaginary semiaxis in the case of symmetric intervals $K_1 = -K_2$, and that it approaches the real positive or negative semiaxis for large or small ratios $m_1/m_2$, respectively.

A remarkable feature in Figure~\ref{fig:example1e} is that the relative error  $|R_n(z)/F(z) - 1|$ stays uniformly small ``above'' the set $K$, i.e., for complex numbers $z$ with positive imaginary part and real part in $K$. We will return to this observation in section~\ref{sec:tensorpml}.

\begin{table}\centering
\caption{Lower and upper error bounds of Theorem~\ref{T2} and actual errors $\max_{z \in K} |R_n(z)/F(z) - 1|$, $F(z) = z^{-1/2}$, for various values of $m = 2n$. The set $K$ is chosen as $K =  [-1\e3,-1]\cup [1,1\e4]$.}
\begin{tabular}{|c|c|c|c|c|c|}
  \hline
  $m$ & $m_1$ & $m_2$ & bound \eqref{T2res2} & relative error &  bound \eqref{T2res1} \\
  \hline
$  6$ & $ 3$ & $ 3$ & $1.22\e-01$ & $3.42\e-01$ & $5.52\e-01$ \\
$ 12$ & $ 5$ & $ 7$ & $8.41\e-03$ & $2.47\e-02$ & $2.85\e-02$ \\
$ 18$ & $ 8$ & $10$ & $5.49\e-04$ & $1.15\e-03$ & $1.83\e-03$ \\
$ 24$ & $11$ & $13$ & $3.57\e-05$ & $8.95\e-05$ & $1.19\e-04$ \\
$ 30$ & $13$ & $17$ & $2.32\e-06$ & $7.01\e-06$ & $7.72\e-06$ \\
$ 36$ & $16$ & $20$ & $1.51\e-07$ & $3.29\e-07$ & $5.02\e-07$ \\
$ 42$ & $19$ & $23$ & $9.79\e-09$ & $2.37\e-08$ & $3.26\e-08$ \\
$ 48$ & $21$ & $27$ & $6.36\e-10$ & $2.01\e-09$ & $2.12\e-09$ \\
$ 54$ & $24$ & $30$ & $4.13\e-11$ & $9.43\e-11$ & $1.38\e-10$ \\
$ 60$ & $27$ & $33$ & $2.69\e-12$ & $6.28\e-12$ & $8.94\e-12$ \\
  \hline
\end{tabular}\label{tab:errors}
\end{table}

\end{example}

\begin{figure}[bth]
\begin{center}
\includegraphics[width=12cm]{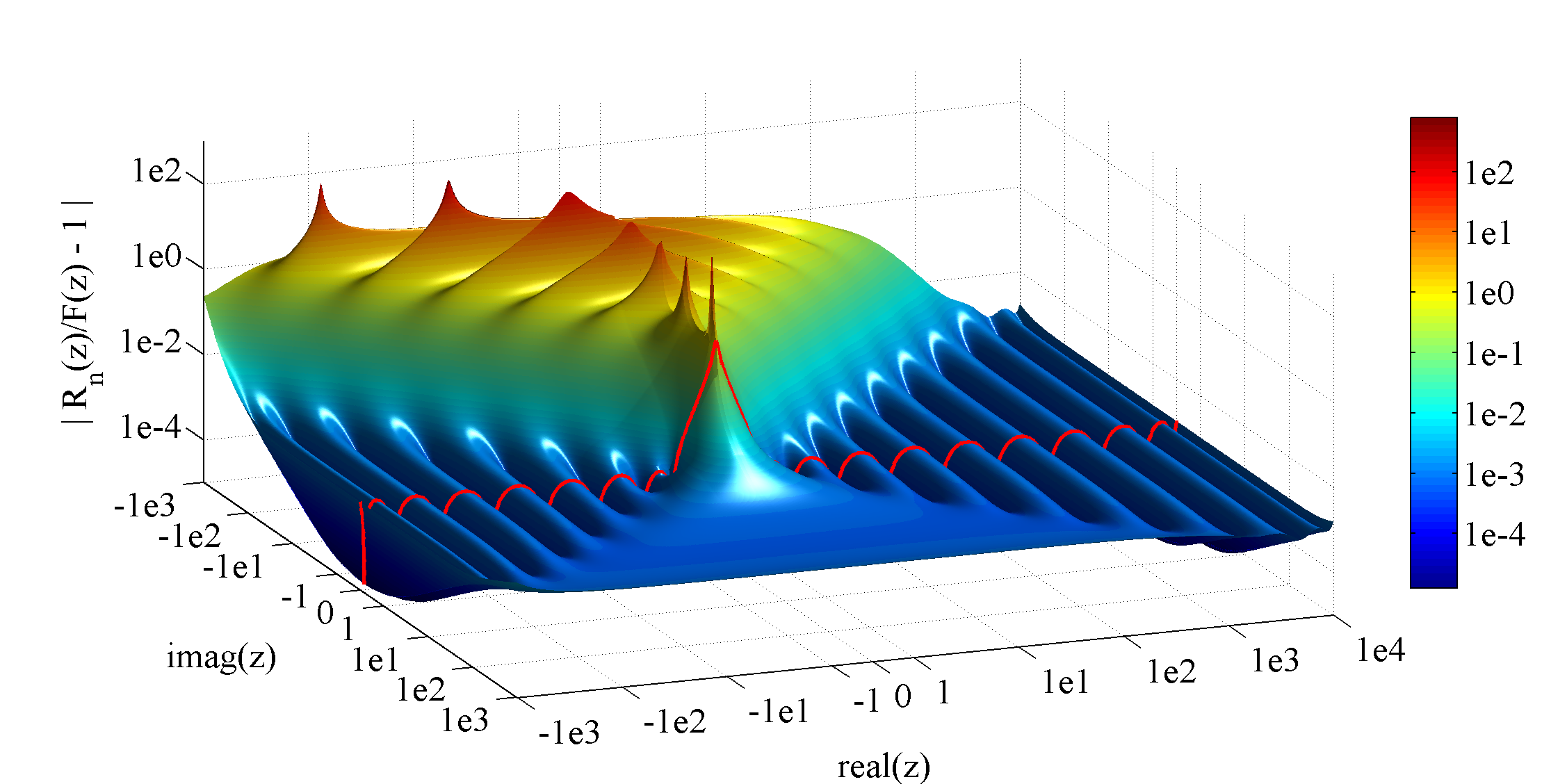}
\end{center}
\caption{Relative error $| R_n(z)/F(z) - 1|$  of a Zolotarev approximant $R_n(z)$ for $K = [-1\e3,-1]\cup [1,1\e4]$ and $n=9$ shown as a logarithmic surface plot over a region in the complex plane. The imaginary axis is plotted in reversed direction for a   better panoramic view.} \label{fig:example1e}
\end{figure}

\section{Finite difference grids from rational approximants}\label{sec:grids}

We  now explain how a rational function $R_n(z)\approx F(z)$
can be transformed into an equivalent \rev{staggered} finite difference grid for \eqref{eq:bvp}.
 Assume that we are given primal grid points and steps
\[
 0 = x_0, \ x_1, \ \ldots, \ x_n, \quad h_j = x_j - x_{j-1},
\]
and dual grid points and steps
\[
 0 = \widehat x_0, \ \widehat x_1,\  \ldots, \ \widehat x_n, \quad \widehat h_{j-1} = \widehat x_j - \widehat x_{j-1},
\]
with $j=1,\ldots,n$ in both cases. Denote by $\mathbf{u}_0,\mathbf{u}_1,\ldots, \mathbf{u}_n$ approximations
to the solution $\mathbf{u}(x)$ of \eqref{eq:bvp}
at the primal grid points $x_0,x_1,\ldots,x_n$. Let the first-order finite differences $(\mathbf{u}_{j} - \mathbf{u}_{j-1})/h_{j}$ be located at the dual points $\widehat x_j$ ($j=1,\ldots,n)$.  We assume that the following finite difference relations
\begin{eqnarray}
 \frac{1}{\widehat h_0}\left( \frac{\mathbf{u}_1 - \mathbf{u}_0}{h_1} + \mathbf{b} \right) &=&  \mathbf{A} \mathbf{u}_0, \label{eq:rel1} \\
 \frac{1}{\widehat h_j}\left(
 \frac{\mathbf{u}_{j+1} - \mathbf{u}_j}{h_{j+1}}  - \frac{\mathbf{u}_{j} - \mathbf{u}_{j-1}}{h_{j}}
        \right) &=&  \mathbf{A} \mathbf{u}_j ,
        \quad   j = 1,\ldots,n-1, \label{eq:rel2}
\end{eqnarray}
are satisfied with the convention that $\mathbf{u}_n = \vnull$. It can be verified by back-substitution
that the value  $\mathbf{u}_0$ specified by these recursive relations can be written as
\[
 \mathbf{u}_0 = R_n(\mathbf{A})\mathbf{b},
\]
where $R_n(z)$ is a rational function of type $(n-1,n)$. By construction, $-R_n(\mathbf{A})^{-1}$ is the Schur complement of the submatrix with positive indices of the system  (\ref{eq:rel1})--(\ref{eq:rel2}).
Written as a finite-length
Stieltjes continued fraction (S-fraction\footnote{We now allow for complex-valued $\widehat h_{j-1}, h_j$ ($j=1,\ldots,n$) in (\ref{eq:contf}), which is different from the classical definition of S-fractions with real positive parameters.}) this function takes the form
\begin{equation}\label{eq:contf}
    R_n(z) =
    \cfrac{1}
    {\widehat h_0 z + \cfrac{1}
                      {h_1 + \cfrac{1}
                             {\widehat h_1 z + \cdots + \cfrac{1}
                                                        {h_{n-1} + \cfrac{1}
                                                                   {\widehat h_{n-1} z + \cfrac{1}{h_n}} }  }}} \ .
\end{equation}
Recalling from above that  the exact solution of $\eqref{eq:bvp}$ satisfies $\mathbf{u}(0) = \mathbf{A}^{-1/2}\mathbf{b}$, we are apparently left with the problem of determining $R_n(z)$ such
that $R_n(\mathbf{A})\mathbf{b}\approx \mathbf{A}^{-1/2}\mathbf{b}$, optimally in some sense.
The conversion of Neumann data $-\mathbf{b}$ to Dirichlet data $\mathbf{u}(0)$
can now be  realized by solving a finite difference relation on a grid
generated from quantities $\widehat h_{j-1}$ and $h_j$ ($j=1,\ldots,n$) in \eqref{eq:contf}.

The connection between the S-fraction (\ref{eq:contf}) and the finite difference problem (\ref{eq:rel1})--(\ref{eq:rel2}) is due to Mark Krein (see, e.g., \cite{KK}).
He viewed the problem (\ref{eq:rel1})--(\ref{eq:rel2}) as a so-called \emph{Stieltjes string}, which is a string of point masses $\widehat h_{j-1}$ and weightless stiffnesses $h_j$ ($j=1,\ldots,n$), both real positive.     There is  a one-to-one correspondence between the set of  Stieltjes strings   and Stieltjes spectral functions  $R_n(z)$, which are  rational functions of type $(n-1,n)$ having $n$ non-coinciding real negative poles and real positive residues. For this case, the S-fraction parameters $\widehat h_{j-1}$ and $h_j$   ($j=1,\ldots,n$) can be computed via $2n$ steps of the Euclidean polynomial division algorithm (see, e.g., \cite{Holtz}), which can be stably executed with the help of the reorthogonalized Lanczos algorithm \cite{DK99}.  The optimal rational approximation of $F(z)$ on a positive real interval is a  Stieltjes problem \cite{IDK00},  hence the generated grid steps are real positive. The approximation problem on a single negative interval can be solved by using $R_n(-z)$, where $R_n(z)$ is the  approximation  on the   symmetrically reflected positive interval. This reflection rotates the grid steps $\widehat h_{j-1}$ and $h_j$   ($j=1,\ldots,n$) by an angle of $\pi/2$ in $\bfC$, i.e., it makes the grid steps purely imaginary. Generally, the problem of optimal approximation on
the union of a positive and a negative interval
leads to non-Stieltjes rational functions $R_n(z)$ of type $(n-1,n)$.
Assuming absence of breakdowns (which are unlikely but can not be definitely excluded), the transformation to the non-Stieltjes rational function (\ref{eq:contf}) can
still be carried out via the complex  $2n$-step  Euclidean  algorithm.
We used the bi-Lanczos extension of the Lanczos-based algorithm \cite{DK99} which, according to our experience, always produced meaningful results.

\begin{example}
We begin with reproducing a real optimal  grid from \eqref{eq:contf} generated for a real positive interval $K=[1,1\e 4]$, see  Figure~\ref{fig:example1c} (left). Similar results were reported in \emph{\cite{IDK00}}. We can consider this example as a degenerate case of the two-interval problem with $m_1=0$ and $m_2=10$. The plot shows ``alternation'' of the primal and dual grid points and monotonically growing steps. The grid looks like
an equidistant grid stretched by a rather smooth transform. It was shown in \emph{\cite{IDK00}} that for large $n$ and small interval ratios such transforms are asymptotically close to the exponential.

%
%

In Figure~\ref{fig:example1c} (right) we plot the complex finite
difference grid points obtained from the continued fraction \eqref{eq:contf} in the case when
$K = [-1\e3,-1] \cup [1,1\e4]$ and $m_1=8$ and $m_2=10$.
We notice the ``alternation'' of the primary and dual points on some ``curve'',
which is an intuitive evidence of a good quality of the grid, i.e., we can speculate that the finite difference solution approximates the exact solution with second-order accuracy on that curve. This curve can be interpreted as the complex PML transform of the real positive axis in accordance with \emph{\cite{Balslev&Combes, ChewWeedon}}.
\end{example}

In summary, we observe that the finite-difference operators on grids
obtained from \eqref{eq:contf} approximate the second-order derivative operator on curves in the complex plane. This can be viewed as a complex extension of Krein's results on the convergence  of the Stieltjes discrete string with impedance $R_n(z)$ to its continuous counterpart with impedance $F(z)$ when  $R_n\rightarrow F$ on $\bfR_+$  \cite{KK}.
Besides internal beauty, this phenomenon may have useful consequences. For example, it lets us hope
that  pseudospectral estimates and  stability results for  continuous PMLs and damped 1D differential operators  \cite{Hagstrom, Joly,DrisTref} remain valid for (\ref{eq:rel1})--(\ref{eq:rel2}) with the optimal grid.

\begin{figure}[bth]
\begin{center}
\includegraphics[width=0.48\textwidth]{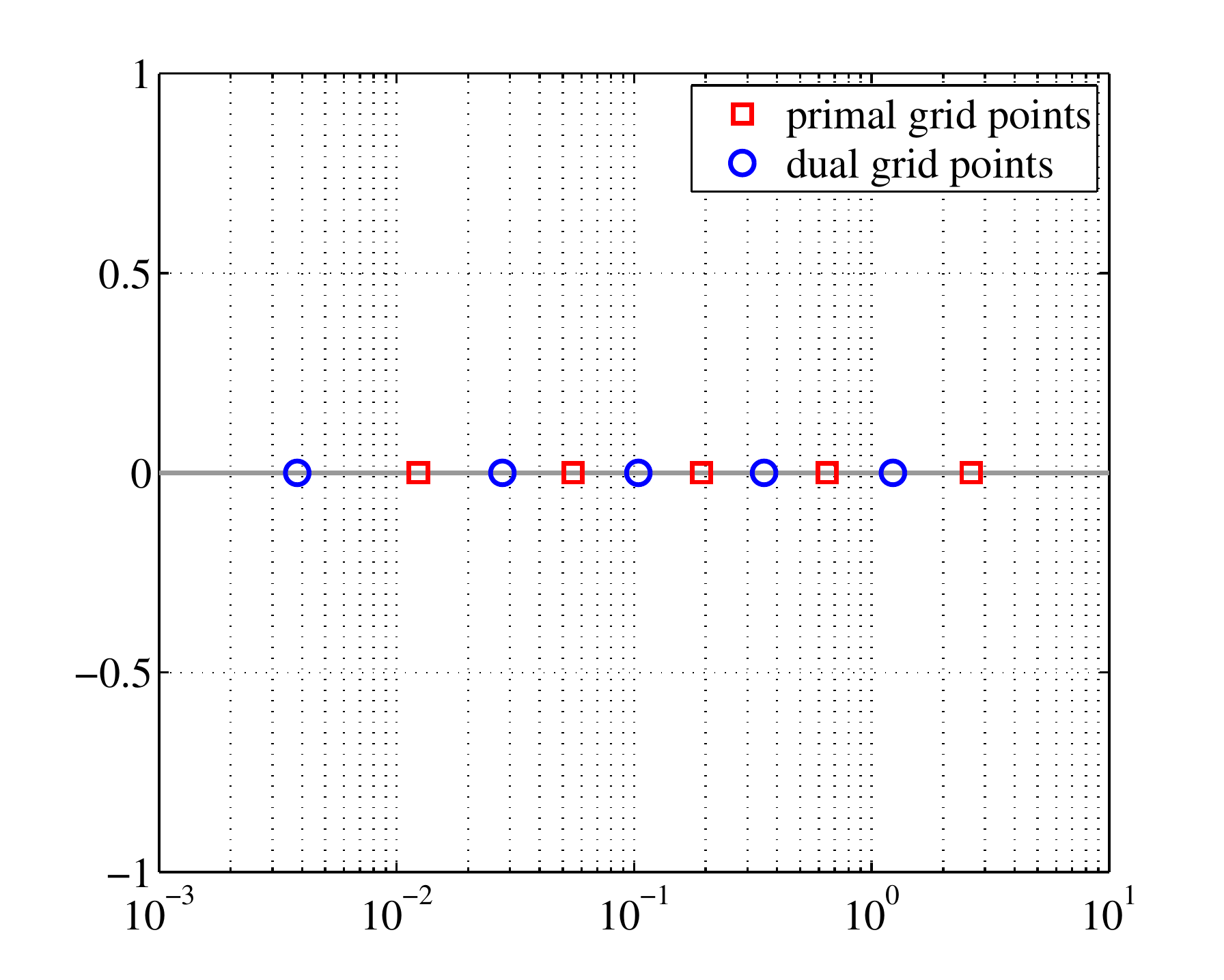}
\includegraphics[width=0.48\textwidth]{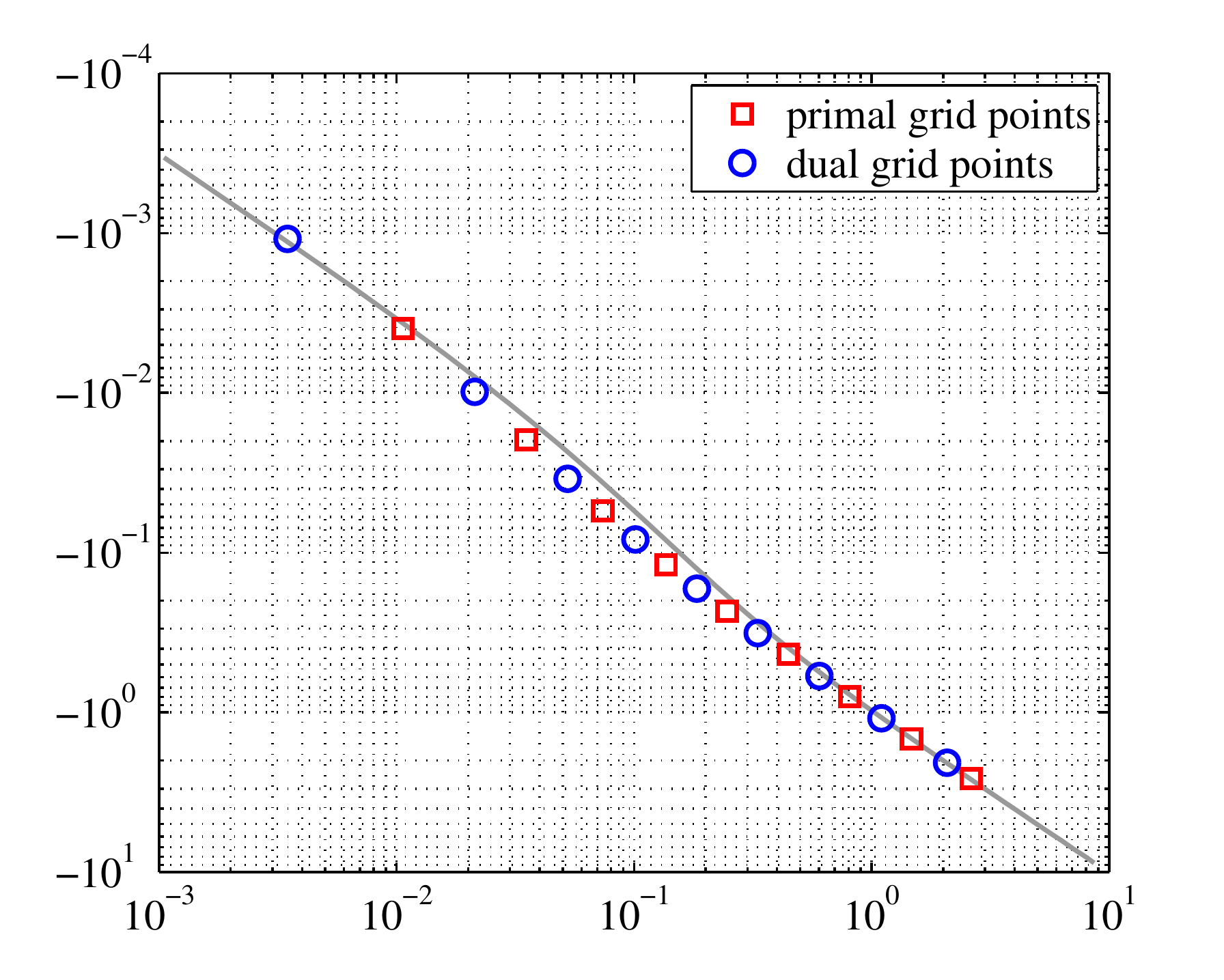}
\end{center}
\caption{Grid points generated from quantities in the continued fraction
\eqref{eq:contf}. Left: In this single-interval case the set $K$ is chosen as $K=[1,1\e4]$ with $m_2=10$ (and $m_1=0$). Right: The set $K$ is chosen as $K = [-1\e3,-1] \cup [1,1\e4]$ with  $m_1=8$ and $m_2=10$. The gray ``continuous'' curve has been obtained by connecting the grid points generated with the parameters $m_1=27$ and $m_2=33$, and we conjecture that the grid points align  on a limit curve as $m\to \infty$.} \label{fig:example1c}
\end{figure}

\section{Summary of the algorithm}\label{sec:algsumm}
In the following we provide a step-by-step description for computing the grid steps $\widehat h_{j-1}$ and $h_j$ ($j=1,\ldots,n$) in \eqref{eq:contf}.
\begin{enumerate}
\item It follows from (\ref{apprerr}) that the numbers
\[
-\left(s^{(\sqrt{-b_1},\sqrt{-a_1})}_j\right)^2, \ \  j=1,\ldots,m_1,
\quad\mbox{and}\quad
\left(s^{(\sqrt{a_2},\sqrt{b_2})}_j\right)^2, \ \  j=1,\ldots,m_2,
\]
are the interpolation nodes for $R_n(z)$ as an interpolant of $F(z)$.
Knowing interpolation nodes and function values,
we compute the coefficients of $P_{n-1}(z)$ and $Q_n(z)$ by means of solving the
corresponding system of linear algebraic equations in high-precision arithmetic.
\item
The poles of the interpolant, i.e., the roots of $Q_n(z)$, can be computed as the eigenvalues of
an associated companion matrix, see \cite[Subsection~7.4.6]{GvL}. To solve this eigenvalue problem  we
use the quasi-version%
\footnote{I.e., we formally use in the complex case the
formulas intended for the real case.}
 of the QR transformation method \cite[\S\,11.6]{NumRecF}
and then, if necessary, correct the roots
by means of a combination of Laguerre's \cite[\S\,9.5]{NumRecF} and
Newton's \cite{Kelley} method.
\item Knowing the poles of $R_n(z)$, the corresponding residues are computed.
\item  Finally, the grid steps $\widehat h_{j-1}$ and $h_j$ ($j=1,\ldots,n$) are computed  using the
recursion formulas \cite[(3.4)]{DK99}, with the underlying
analogue of an inverse eigenvalue problem for a symmetric
tridiagonal matrix (see \cite[subsection~3.1, item~3$^{\circ}$]{DK99},
\cite[theorem~7.2.1]{Parlett}) being solved by a quasi-Lanczos process \cite[Ch.~6]{CullumW85} with
quasi-reorthogonalization. Here we used the well-known connection between the Lanczos and Euclidean algorithms (see, e.g., \cite{Holtz}).
\end{enumerate}

\section{Adaptation to a second-order finite difference framework}\label{sec:impl}

\subsection{Approximation of the discrete impedance function}\label{sec:discr}

So far we have considered the function $F(z)=z^{-1/2}$, which arises when solving the
boundary-value problem \eqref{eq:bvp} for $x\in [0,+\infty)$. When this problem is seen as an infinite extension of some interior computational domain, the exponential convergence of the interpolant $R_n(z)$
is consistent with a high-order (or even  spectral) discretization of the operator acting in this computational domain.

 However, it is also possible
to compute the NtD map of a discretized version of \eqref{eq:bvp} on a uniform infinite grid via rational
approximation of a slightly modified function $F_h(z)$ to be determined below. This function will lead to a three-term finite difference scheme which is appropriate for being combined with a standard second-order finite difference discretization in the interior computational domain, because it allows for the  elimination of spurious reflections from the PML boundary due to the error of the interior discretization.

Given a fixed step size $h>0$, let us consider the problem (\ref{eq:rel1})--(\ref{eq:rel2}) on the infinite equidistant grid  with $\widehat h_0=0.5 h$ and $\widehat h_j=h_j = h$ for $j=1,\ldots,\infty$.
We will determine a function $F_h(z)$ such that
\[\mathbf{u}_0 = F_h(\mathbf{A})\mathbf{b} \]
via a well-known  approach widely used in the representation of irrational numbers via continued fractions (see, e.g., \cite[section~9]{EL}).   This approach was already applied in \cite{ThGu}  to the infinite lattice problem: the
 infinite-length S-fraction representation of $F_h$
analogous to \eqref{eq:contf} is
\[
    R_h(z) =
    \cfrac{1}
    {0.5  h z + \cfrac{1}
                      {h + \cfrac{1}
                             {h z + \rev{\cfrac{1}
                             {h + \ddots}}   }}}
\]
(for a proof of convergence we refer to \cite{Stieltjes1894} or
\cite[theorem~4.58]{JonesThron}).
The remainder continued fraction
\[
S(z)=
\cfrac{1}{h +
 \cfrac{1}{hz +
  \cfrac{1}{h+
   \cfrac{1}{hz + \ddots   }}}}
\]
evidently satisfies the equation
\[
S(z)=
\cfrac{1}{h +
 \cfrac{1}{hz+S(z)}}\, ,
\]
or equivalently $S(z)^2+hzS(z)-z=0$. Since $0.5hz+S(z)=R_n(z)^{-1}=F_h(z)^{-1}$, we
have arrived at the quadratic equation
\[F_h(z)^2 = \frac{1}{z + (0.5 hz)^2}.\]
We choose the root which converges to the exact impedance $F(z)$ as $h\rightarrow 0$, i.e.,
\begin{equation}\label{eq:discr}
    F_h(z) = \frac{1}{\sqrt{z + (0.5 hz)^2}}.
\end{equation}
This function, which we will refer to as the \emph{discrete impedance function}, approximates with second-order accuracy the exact impedance at the boundary, so being centered, the resulting finite difference scheme is of  second order globally.

Analogously to what we had achieved with \eqref{eq:invsqrtmv} for continuous $x$, the relation \eqref{eq:discr} allows us
to convert the Neumann data $-\mathbf{b}$ at $x=0$ into the Dirichlet data $\mathbf{u}_0$ without
actually solving the infinite lattice problem.

{For a given $h>0$ let us define $\s=\frac{h^2}{4}$.   
The invertible linear fractional change of variables
\be \label{zw}
w=\frac{z}{\s z+1}
\ee
translates the union of a negative and a positive segment $K=[a_1,b_1]\cup [a_2,b_2]$ again
into the union of a negative and a positive segment. Let us assume\footnote{As discussed earlier, the parameter $a_1$ should be set to a lower bound  of $\mathbf{A}$'s spectral interval, in which case the condition $-\s^{-1}<a_1$ corresponds to the Nyquist sampling criterion of two grid points per wave length. This assumption should be met by any reasonable  discretization scheme.} that $-\s^{-1}<a_1$. 
Let $P_{n-1}/Q_n$  denote the rational approximant of theorem~\ref{T2}
for the image of $K$ under transformation~(\ref{zw}). Then
\beas
\left| \sqrt{w}\cdot\frac{P_{n-1}(w)}{Q_n(w)}-1 \right|
= \left| \sqrt{\frac{z}{\s z+1}}\cdot
\frac{P_{n-1}(\frac{z}{\s z+1})}{Q_n(\frac{z}{\s z+1})}-1 \right| \\
= \left| \sqrt{\frac{z}{\s z+1}}\cdot
\frac{P_{n-1}(\frac{z}{\s z+1})(\s z+1)^{n}}
{Q_n(\frac{z}{\s z+1})(\s z+1)^{n}}-1 \right|
= \left| F_h(z)\cdot
\frac{P_{n-1}(\frac{z}{\s z+1})(\s z+1)^{n-1}}
{Q_n(\frac{z}{\s z+1})(\s z+1)^{n}}-1 \right|
\eeas
is small on $K$, the numerator and the denominator
\[
P_{n-1}\left(\frac{z}{\s z+1}\right)(\s z+1)^{n-1}, \qquad
Q_n\left(\frac{z}{\s z+1}\right)(\s z+1)^{n}
\]
being polynomials of degrees $\le n-1$ and $\le n$, respectively.
\SG{We have thereby established a direct relation between the errors 
of the rational interpolants for $F(z)$
and $F_h(z)$ on transformed compact sets, respectively, with the interpolation nodes
being transformed accordingly. This allows us to conclude that we obtain 
identical convergence rates for both interpolation processes. \SG{In particular, 
Theorem~\ref{T2} holds with $F(z)$ being replaced by $F_h(z)$.}
}} 

We would like to mention that a rational approximation-based absorbing boundary condition for the infinite lattice was suggested  in \cite{ThGu} and combined with a trapezoidal finite element approach in  \cite{GT00}. However, that approach required a modification of the Helmholtz equation by a higher-order term. On the contrary, in our framework the discreteness can be incorporated simply by adjusting  the PML grids. Visually these grids look very similar to the ones shown in Figure~\ref{fig:example1c}, i.e.,  we can speculate again that they approximate the exact solution $\mathbf{u}(x)$ of \eqref{eq:bvp} with second-order accuracy on some modified $x$-curve in the complex plane.

\subsection{Matching interior and exterior discretizations via a single grid}\label{sec:patch}

Let us consider the second-order infinite equidistant finite difference problem
\begin{equation}\label{single}
 \frac{1}{h}
 \left( \frac{\mathbf{u}_{j+1} - \mathbf{u}_j}{h} -  \frac{\mathbf{u}_{j} - \mathbf{u}_{j-1}}{h} \right)
 - \mathbf{A}\mathbf{u}_j=\mathbf{q}_j, \qquad j = -\ell,\ldots,-1,0,1,\ldots,\infty\end{equation} with boundary conditions
\begin{equation}\label{gbc}
\mathbf{u}_{-\ell-1}=\vnull, \quad \lim_{j\to\infty} \mathbf{u}_{j}=\vnull,
\end{equation}
assuming $\mathbf{q}_j=\vnull$ for $j\ge 0$.
  Problem (\ref{single}) can be split equivalently into an interior finite-dimensional system
\begin{eqnarray}\label{inter}
 \frac{1}{h}
 \left( \frac{\mathbf{u}_{j+1} - \mathbf{u}_j}{h} -  \frac{\mathbf{u}_{j} - \mathbf{u}_{j-1}}{h} \right)
 - \mathbf{A}\mathbf{u}_j&=&\mathbf{q}_j, \qquad j = -\ell,\ldots,-1,\\
\frac{1}{0.5h}\left(-\mathbf{b}-\frac{\mathbf{u}_{0} - \mathbf{u}_{-1}}{h}\right)
 - \mathbf{A}\mathbf{u}_0&=&\vnull, \nonumber
\end{eqnarray}
and an exterior infinite  system
\begin{eqnarray}\label{exter}
\label{interbc}
\frac{1}{0.5h}\left(\frac{\mathbf{u}_{1} - \mathbf{u}_{0}}{h}+\mathbf{b}\right)
 - \mathbf{A}\mathbf{u}_0&=&\vnull, \\
 \frac{1}{h}
 \left( \frac{\mathbf{u}_{j+1} - \mathbf{u}_j}{h} -  \frac{\mathbf{u}_{j} - \mathbf{u}_{j-1}}{h} \right)
 - \mathbf{A}\mathbf{u}_j&
 =&{\vnull}, \qquad j = 1,\ldots,\infty,\nonumber
\end{eqnarray}
both systems being coupled via a vector variable $\mathbf{b}$.\footnote{Problem (\ref{single})--(\ref{gbc})  can be viewed as the second-order discretization of
$
    \frac{\partial^2}{\partial x^2} \mathbf{u} -\mathbf{A} \mathbf{u}=\mathbf{q}$,
     $\mathbf{u} \big|_{x=-h(\ell+1)} = \vnull, \  \mathbf{u} \big|_{x=+\infty} = \vnull$ for some regular enough $\mathbf{q}$ supported on $[-h(\ell+1),0]$.
 As the infinite exterior problem (\ref{exter}) approximates with  second-order accuracy the same equation on $[0,+\infty)$ with conditions  $\mathbf{u} \big|_{x=0} = -\mathbf{b}$ and $\mathbf{u} \big|_{x=+\infty} = \vnull$, the relation
 (\ref{inter}) approximates with second order the same equation restricted to $[-h(\ell+1),0]$ with conditions
$\mathbf{u} \big|_{x=-h(\ell+1)} = \vnull$ and $\mathbf{u} \big|_{x=0} = -\mathbf{b}$.}

Problem (\ref{exter})  (with the
condition at infinity) was already considered in section~\ref{sec:discr}, and can be exactly eliminated using the discrete impedance function  (\ref{eq:discr}),
\begin{eqnarray*}
 \frac{1}{h}
 \left( \frac{\mathbf{u}_{j+1} - \mathbf{u}_j}{h} -  \frac{\mathbf{u}_{j} - \mathbf{u}_{j-1}}{h} \right)
 - \mathbf{A}\mathbf{u}_j&=&\mathbf{q}_j, \qquad j = -\ell,\ldots,-1,\\
\frac{1}{0.5h}\left(-F_h(\mathbf{A})^{-1}\mathbf{u}_{0} -\frac{\mathbf{u}_{0} - \mathbf{u}_{-1}}{h}\right)
 - \mathbf{A}\mathbf{u}_0&=&\vnull, \qquad \mathbf{u}_{-\ell-1}=\vnull.
\end{eqnarray*}
This formally corresponds to a Schur complement.
Upon substitution $R_n(\mathbf{A})\approx F_h(\mathbf{A})$ we arrive at the approximate problem
\begin{eqnarray*}
 \frac{1}{h}
 \left( \frac{\mathbf{u}^n_{j+1} - \mathbf{u}^n_j}{h} -  \frac{\mathbf{u}^n_{j} - \mathbf{u}^n_{j-1}}{h} \right)
 - \mathbf{A}\mathbf{u}^n_j&=&\mathbf{q}_j, \qquad j = -\ell,\ldots,-1,\\
\frac{1}{0.5h}\left(-R_n(\mathbf{A})^{-1}\mathbf{u}^n_{0} -\frac{\mathbf{u}^n_{0} - \mathbf{u}^n_{-1}}{h}\right)
 - \mathbf{A}\mathbf{u}^n_0&=&\vnull, \qquad \mathbf{u}^n_{-\ell-1}=\vnull.
\end{eqnarray*}
Hence \[\|\mathbf{u}_j^n-\mathbf{u}_j \|=O(\|R_n(\mathbf{A})- F_h(\mathbf{A})\|), \]
 since all the involved linear systems are well posed
uniformly in $n$.

 Performing similar manipulations with the approximate problem in reverse order, we
obtain the equivalent system (\ref{intern})--(\ref{extern})
\begin{eqnarray}\label{intern}
 \frac{1}{h}
 \left( \frac{\mathbf{u}^n_{j+1} - \mathbf{u}^n_j}{h} -  \frac{\mathbf{u}^n_{j} - \mathbf{u}^n_{j-1}}{h} \right)
 - \mathbf{A}\mathbf{u}^n_j&=&\mathbf{q}_j, \qquad j = -\ell,\ldots,-1,\\
\frac{1}{0.5h}\left(-\mathbf{b}-\frac{\mathbf{u}^n_{0} - \mathbf{u}^n_{-1}}{h}\right)
 - \mathbf{A}\mathbf{u}^n_0&=&\vnull, \nonumber
\end{eqnarray}

\begin{eqnarray}
 \frac{1}{\widehat h_0}\left( \frac{\mathbf{u}^n_1 - \mathbf{u}^n_0}{h_1} + \mathbf{b} \right) &-&  \mathbf{A} \mathbf{u}^n_0=\vnull, \label{extern} \\
 \frac{1}{\widehat h_j}\left(
 \frac{\mathbf{u}^n_{j+1} - \mathbf{u}^n_j}{h_{j+1}}  - \frac{\mathbf{u}^n_{j} - \mathbf{u}^n_{j-1}}{h_{j}}
        \right) &-&  \mathbf{A} \mathbf{u}^n_j=\vnull,
        \qquad   j = 1,\ldots,n-1, \nonumber
\end{eqnarray}
by introducing $\mathbf{b}$ and fictitious variables $\mathbf{u}^n_j$ with positive subindices  which, unlike their negative counterparts, do not approximate corresponding components of $\mathbf{u}(x)$.
Finally, eliminating $\mathbf{b}$ we can merge the systems (\ref{intern})--(\ref{extern}) into a single recursion
 \begin{eqnarray}\label{total}\nonumber
 \frac{1}{\widehat h_j}\left(
 \frac{\mathbf{u}^n_{j+1} - \mathbf{u}^n_j}{h_{j+1}}  - \frac{\mathbf{u}^n_{j} - \mathbf{u}^n_{j-1}}{h_{j}}
        \right) -\mathbf{A} \mathbf{u}^n_j ={\mathbf{q}_j},
        \qquad   j = -\ell,\ldots,n-1, \\ \mathbf{u}^n_{-\ell}=\vnull, \ \mathbf{u}^n_{n}=\vnull,\nonumber
\end{eqnarray}
with the convention that $\widehat h_j := h$ for $j<0$, $h_j :=h$ for $j\le 0$,
$\widehat h_j := \widehat h_j$ for $j>0$, $h_j :=h_j$ for $j> 0$, and
$\widehat h_0 := \widehat h_0 + h/2$ (see also Figure~\ref{fig:grid}).
This finite difference scheme is easy to implement by simply modifying the $n$ trailing primal and dual grid steps in a given finite difference scheme with step size $h$.
We reiterate that this scheme converges exponentially with error $O(\|R_n(\mathbf{A})- F_h(\mathbf{A})\|)$ to the solution of (\ref{single})--(\ref{gbc}) in the interior domain, i.e., for the nonpositive subindices.

The above derivation can easily be extended to variable operators $\mathbf{A} = \mathbf{A}_j$ in the interior domain and tensor-product PML discretizations.
 This will be illustrated by a numerical example in section~\ref{sec:tensorpml}.

\begin{figure}[h]
\begin{center}
\includegraphics[width=0.75\textwidth]{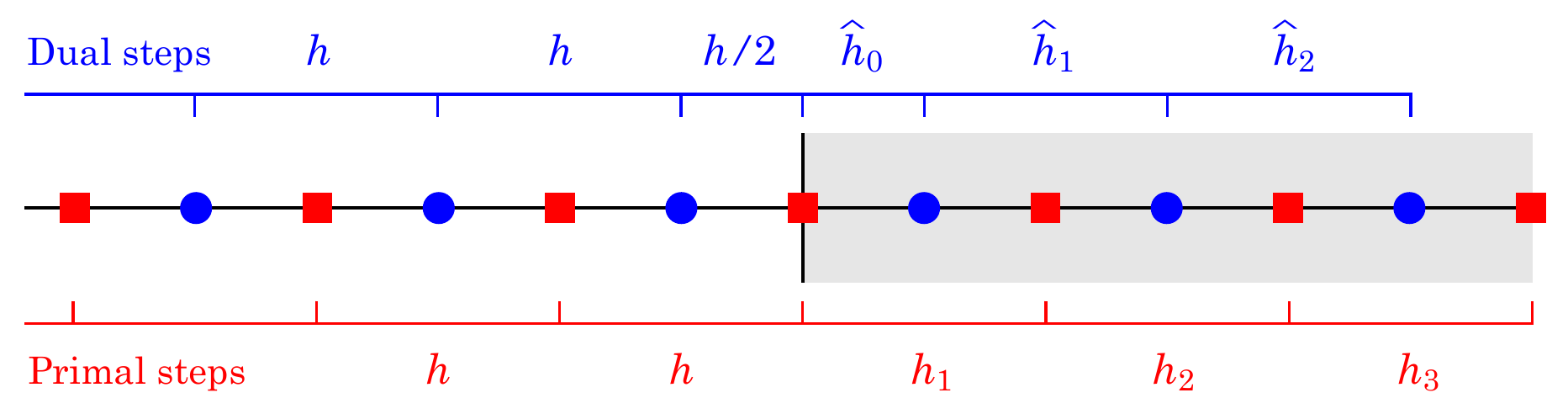}
\end{center}
\caption{Schematic view of a finite difference grid appended with an absorbing boundary layer  generated from
quantities in the continued fraction \eqref{eq:contf}. The example shown here is for the case $n=3$.
The gray-shaded region corresponds to the appended absorbing boundary layer, and the grid steps  $\widehat h_0,\widehat h_1,\ldots, \widehat h_{n-1}$ and $h_1,h_2,\ldots,h_n$ in this layer are generally complex.} \label{fig:grid}
\end{figure}

\section{Numerical experiments}\label{sec:numex}
\subsection{Waveguide example}\label{sec:waveguide}
To test the accuracy of our absorbing boundary layer, we consider the inhomogeneous Helmholtz equation
\[
    \Delta u(x,y) + k^2 u(x,y) = f(x,y)
\]
on a rectangular domain $\Omega = [ 0,L] \times [0,H]$ of length $L$ and height $H$. We prescribe homogeneous Dirichlet conditions at the upper and lower boundaries in $y$.
The source term is set to
\[
    f(x,y) = 10\cdot \delta(x - 511\pi/512)\cdot \delta(y - 50\pi/512),
\]
with the Dirac delta function $\delta(\cdot)$.
%
%
%
%
%

Our aim is to verify that our absorbing boundary layer models the correct physical behavior.
To this end we solve the above Helmholtz equation on two rectangular domains with fixed height $H=\pi$ and different lengths $L=\pi$ and $L=2\pi$, respectively. See also Figure~\ref{fig:example11} (left and right, respectively).
The wave number is chosen as $k=50$. The problem is discretized by central finite differences with step size $h=\pi/512$ in both coordinate directions. The eigenvalues of the resulting tridiagonal matrix $\mathbf{A}$, corresponding to the operator $-\partial^2/\partial y^2 - k^2$ on $[0,\pi]$ with homogeneous Dirichlet boundary conditions, are explicitly known and  eigenvalue inclusion intervals are
\[
    [a_1,b_1]\cup [a_2,b_2] = [-2.50\e3,-1.95\e1]\cup [7.98\e1,1.04\e5].
\]
We extend the interior finite difference grid by our absorbing boundary layer with $n=m/2$ additional grid points to the left of $x=0$ and to the right of $x=L$, with the near-optimal grid steps computed from a rational interpolant $R_n(z)$ of $F_h(z)$ as explained in section~\ref{sec:discr}. The physical domain can hence be thought of as an infinite strip parallel to the $x$-axis.
We therefore expect the solutions of both problems (with $L=\pi$ and $L=2\pi$) to coincide when they are restricted to $[0,\pi]\times [0,\pi]$. Visually, this is indeed the case, as one can see in Figure~\ref{fig:example11} (where $n=10$).
Note how the amplitude of the solution is damped very quickly inside the absorbing boundary layer.

To quantify the accuracy of our absorbing boundary layer numerically, we plot in Figure~\ref{fig:example11e} the relative uniform norm of the difference of the two numerical solutions $u_1(x,y)$ and $u_2(x,y)$ restricted to $[0,\pi]\times [0,\pi]$, i.e.,
\begin{equation}\label{eq:err}
    \mathrm{err} = \max_{\substack{0\leq x,y\leq \pi}}
    | u_1(x,y) - u_2(x,y) |\Big/ \max_{\substack{0\leq x,y\leq \pi}} | u_1(x,y)|.
\end{equation}
Indeed, this figure reveals exponential convergence with the rate $\rho$ given in Theorem~\ref{T2}. In this example, the expected rate is $\rho\approx 0.57$ and this is indicated by the slope of the dashed line in Figure~\ref{fig:example11e}.

We would like to mention that absorbing boundary layers usually require some physical separation from the support of the right-hand side (the source term) \cite{Hagstrom2}. However, thanks to the efficient   absorption of evanescent  and propagative modes even on spectral subintervals with extreme interval ratios,  we are able to place our Dirac source  extremely close to the PML boundary (only one grid point away, see the right of Figure~\ref{fig:example11}) without deteriorating convergence (see Figure~\ref{fig:example11e}).

\begin{figure}[t]
\hspace*{-7mm}\includegraphics[width=0.695\textwidth]{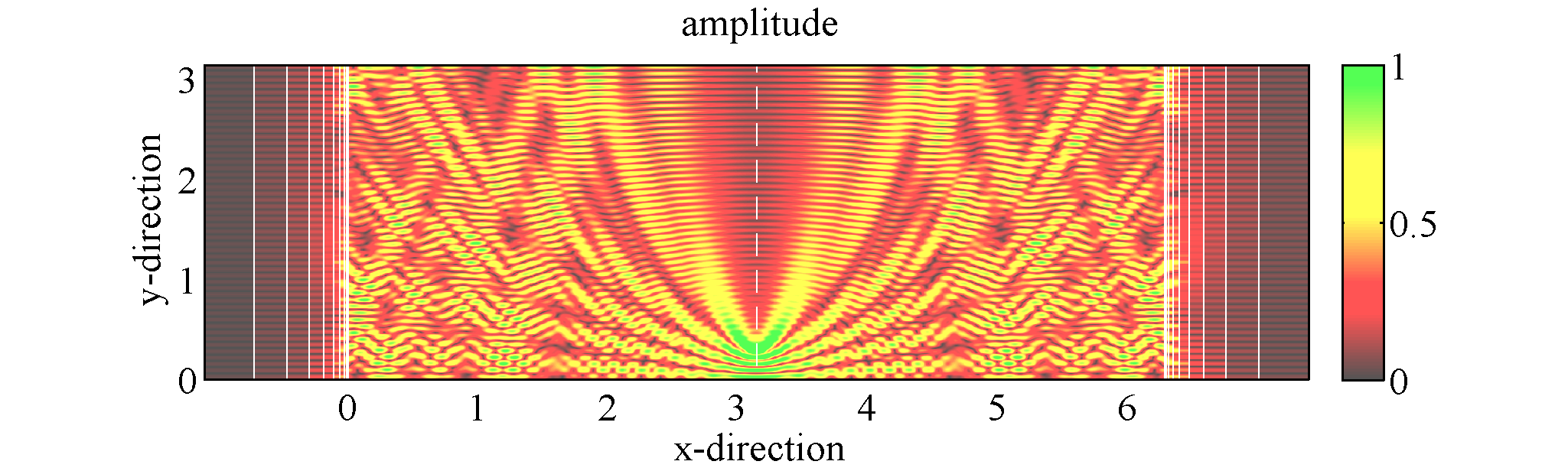}\hspace*{-8mm}
\includegraphics[width=0.462\textwidth]{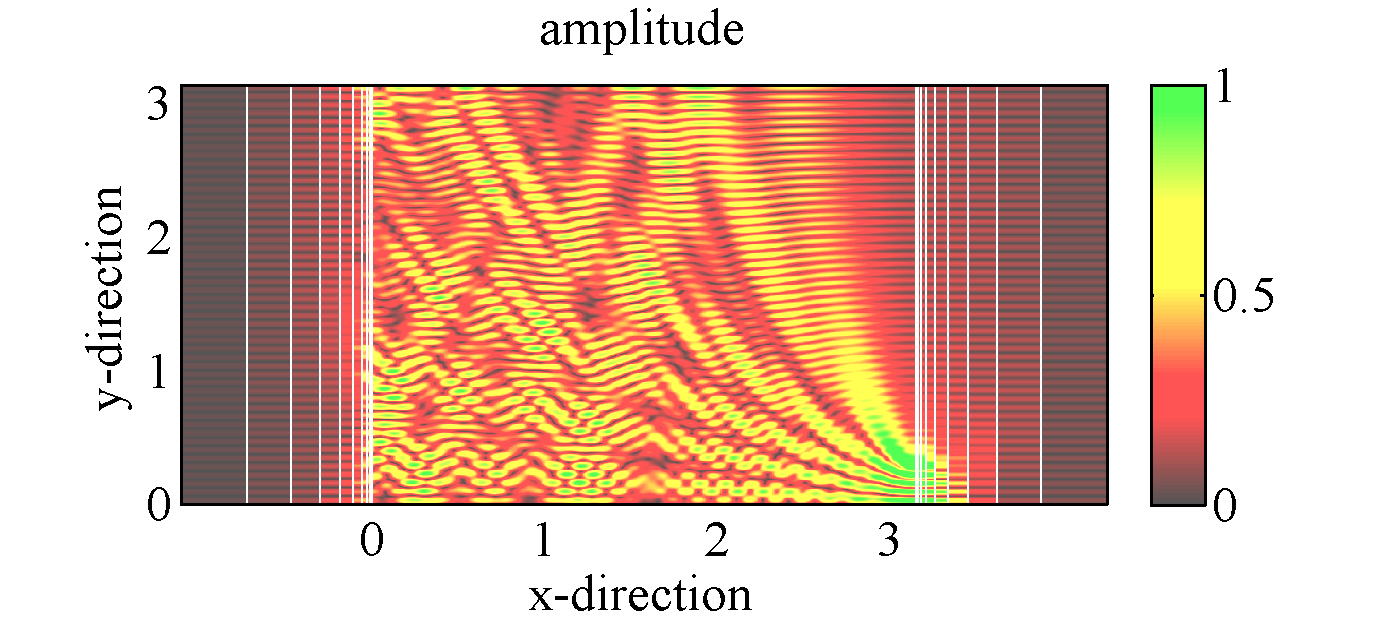}\\[3mm]
\hspace*{-7mm}\includegraphics[width=0.695\textwidth]{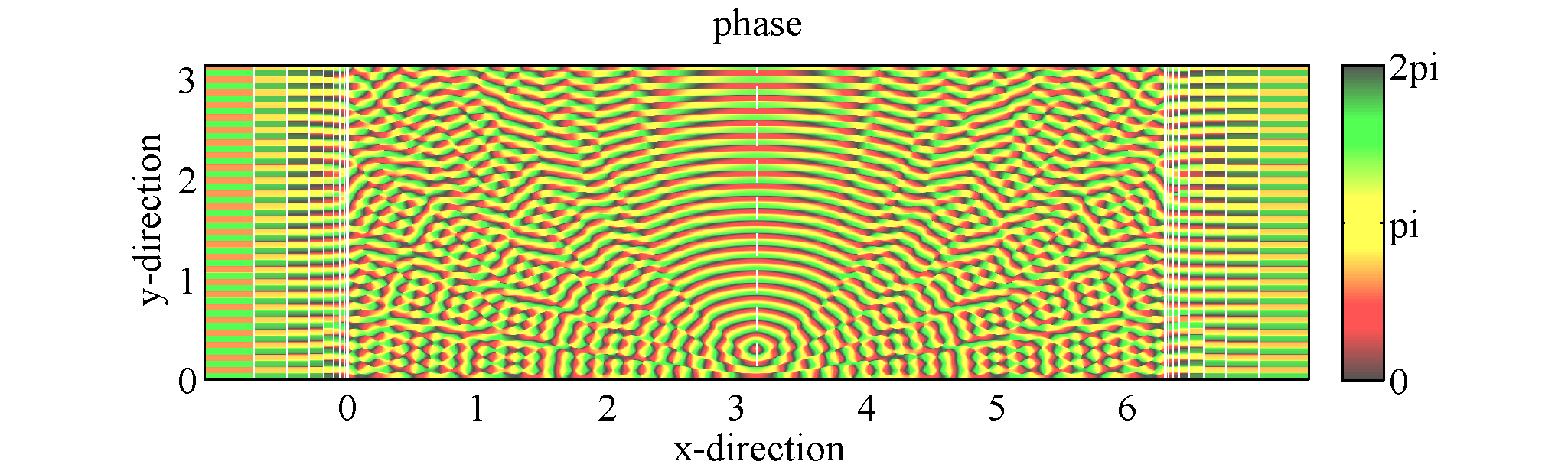}\hspace*{-8mm}
\includegraphics[width=0.462\textwidth]{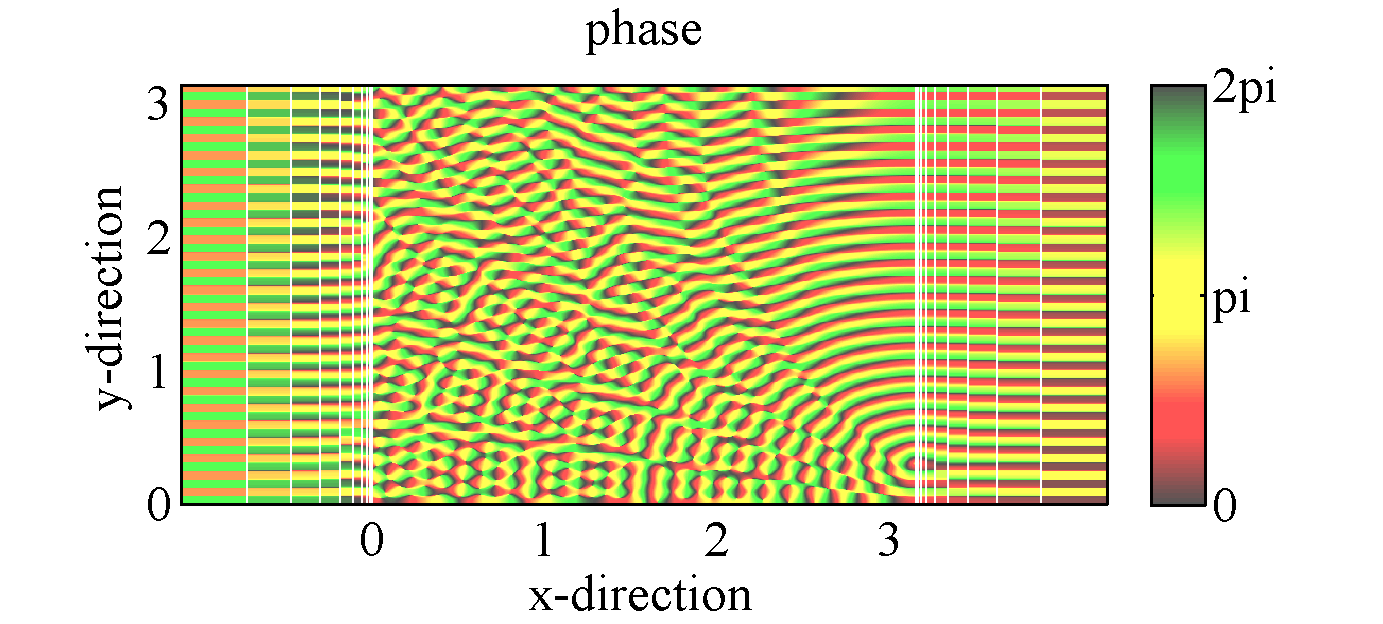}
\caption{Amplitude (top) and phase (bottom) of the solution to the waveguide problem in section~\ref{sec:waveguide} on
two rectangular domains (left/right) which differ in their length. The left domain is of length $L=2\pi$ in the $x$-direction, whereas the right domain is of length $L=\pi$. Both domains have been appended with absorbing layers at the left and right boundaries. As the absorbing boundary layers serve the purpose of extending the physical domain towards infinity, both solutions are expected to coincide on the restriction to $x\in [0,\pi]$. In these pictures we have chosen $m=20$, so there are $n=10$ points appended to the left and right boundaries. The step size in the interior domain is $h=\pi/512$ in both coordinate directions.} \label{fig:example11}
\vspace*{5mm}
\begin{center}
\includegraphics[width=0.6\textwidth]{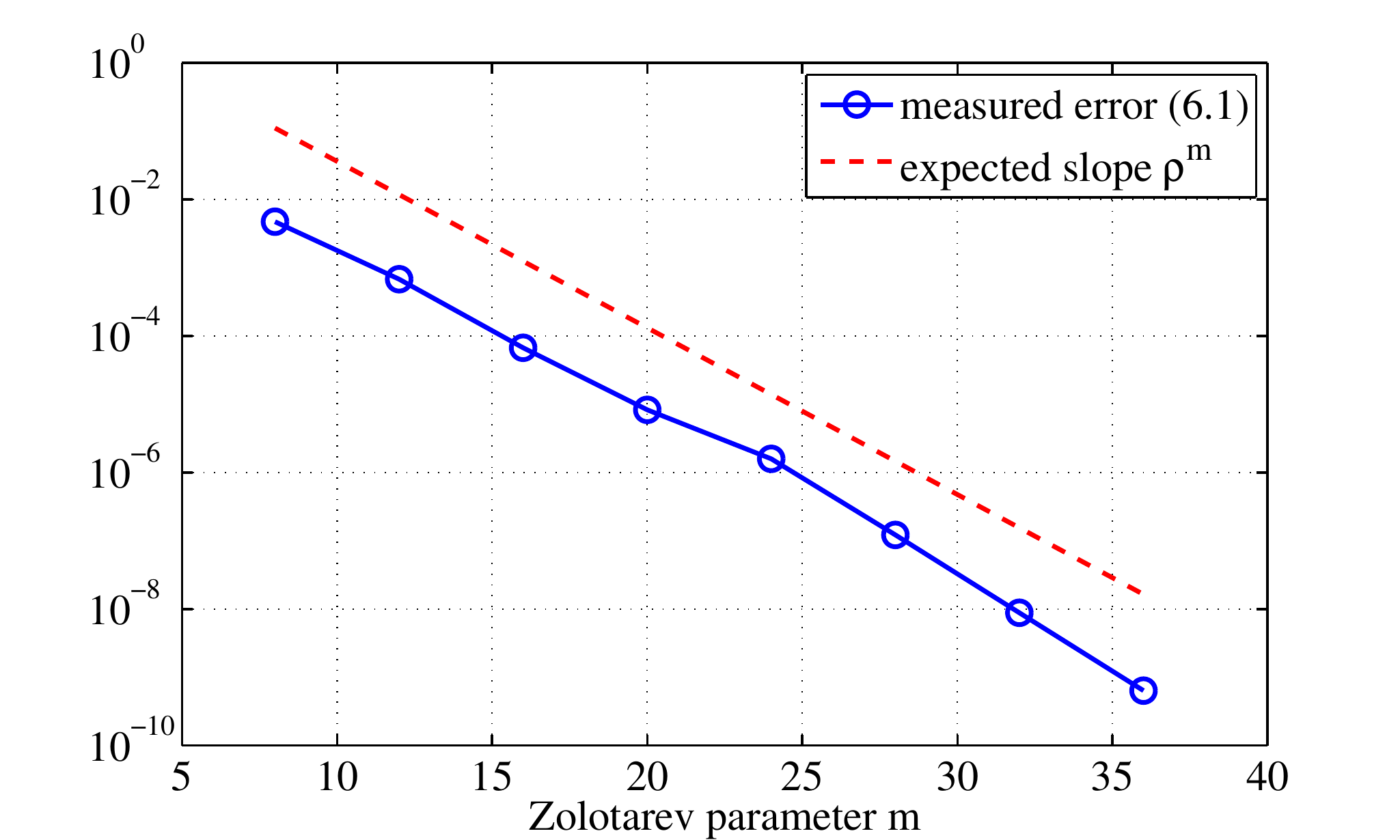}
\end{center}
\caption{Exponential convergence of the accuracy of the absorbing boundary layers for the waveguide problem in section~\ref{sec:waveguide}
with varying Zolotarev parameter $m\in\{8,12,\ldots,36\}$ (twice the number of grid points in each  absorbing boundary layer).
The expected convergence rate $\rho\approx 0.57$ by Theorem~\ref{T2} is indicated by the dashed line.} \label{fig:example11e}
\end{figure}

%
%

\subsection{PML in multiple coordinate directions}\label{sec:tensorpml}

In this experiment we demonstrate how our perfectly matched layer can be
used to mimic domains which are unbounded in several coordinate directions, \rev{and where there is a nonconstant wave speed. 
To this end consider 
\[
    c(x,y)^2\Delta u(x,y) + k^2 u(x,y) = f(x,y)
\]
on a square domain $\Omega_1 = [0,1]^2$. The wave speed $c(x,y)$ varies as indicated in Figure~\ref{fig:drawing}, with $c=1$ in the gray region (background material), $c=1/\sqrt{2}$ in the light gray layer, and $c=\sqrt{2}$ in the dark inclusion at the bottom (see also  Figure~\ref{fig:example12}).  The wave number is
 chosen as $k=120$ and the source term is set to
 \[
    f(x,y) = \delta(x - 120/400)\cdot \delta(y - 280/400).
\]
The domain $\Omega_1$ is discretized by central finite differences with step size $h=1/400$ in both coordinate directions. We aim to append absorbing boundary layers with $n\in\{7,9,11,13\}$ grid points
at each of the four edges of $\Omega_1$.

 For constructing the absorbing layers in the $y$-direction (below $y=0$ and above $y=1$) we need inclusion intervals for the negative and positive eigenvalues of $\mathbf{L}_x - k^2 \mathbf{I}$, where $\mathbf{L}_x$ is the discretization of $-c(x)^2\partial^2/\partial x^2$ on $[0,1]$ with homogeneous Dirichlet boundary and $c(x)=1$ for $y\in\{0,1\}$. (Note that $c(x,y)$ varies only tangentially along the boundaries of $\Omega_1$, so for $y\in \{0,1\}$ we can indeed write $c(x,y)=c(x)$.) Possible inclusion intervals for the eigenvalues are
 \begin{equation}\label{eq:interv}
  [a_1,b_1] \cup [a_2,b_2] = [-1.44\e4,-2.53\e2]\cup [4.94\e2,6.26\e5].
 \end{equation}
For constructing the absorbing layers in the $x$-direction (to the left of $x=0$ and to the right of $x=1$) we need  inclusion intervals \eqref{eq:interv} for the negative and positive eigenvalues of $\mathbf{L}_y - k^2 \mathbf{I}$, where $\mathbf{L}_y$ is the finite difference discretization of $-c(y)^2\partial^2/\partial y^2$ on $[0,1]$ with homogeneous Dirichlet boundary and
\[
 c(y) = \begin{cases}
1/\sqrt{2}, & 0.255 \leq y \leq 0.44, \\
1, & \text{otherwise.}
\end{cases}
\]
Possible intervals are
 \begin{equation}\label{eq:interv2}
  [a_1,b_1] \cup [a_2,b_2] = [-1.44\e4,-2.42\e2]\cup [4.82\e2,6.26\e5].
 \end{equation}

From the union of intervals in \eqref{eq:interv} and \eqref{eq:interv2} we can now calculate the grid steps of absorbing boundary layers in the $y$- and $x$-directions, and then modify the finite difference matrices to $\mathbf{\widehat L}_x$ and $\mathbf{\widehat L}_y$, respectively. As in the previous example, this is done by computing a rational interpolant $R_n(z)$ of $F_h(z)$ defined in section~\ref{sec:discr}.}

However, there is a small subtlety one has to be aware of with the  approach just described: effectively, the NtD operators are now given as $F_h(\mathbf{\widehat L}_x - k^2 \mathbf{I})$ and
$F_h(\mathbf{\widehat L}_y - k^2 \mathbf{I})$, respectively, and the involved matrices are no longer Hermitian. In Figure~\ref{fig:example12cd} (left) we show the eigenvalues of $\mathbf{\widehat L}_x$ and $\mathbf{\widehat L}_y$, respectively, and observe that these eigenvalues have ``lifted off'' the real axis into the upper half of the complex plane, in agreement with the analysis of \cite{DrisTref}  for continuous one-dimensional damped operators. From  Figure~\ref{fig:example1e} we find at least visually that the Zolotarev approximant is of a good quality in this region as well, and the accuracy of the resulting absorbing boundary layers  should still be satisfactory.

To quantify the accuracy numerically, we solve the same Helmholtz problem on a smaller domain $\Omega_2 = [0.1,0.9]^2$, again appended with absorbing boundary layers of $n$ grid points at each of the
four edges of $\Omega_2$. As the source term $f(x,y)$ is supported inside $\Omega_2$, we expect coinciding solutions $u_1(x,y)$ and $u_2(x,y)$ on their restrictions to~$\Omega_2$.
In Figure~\ref{fig:example12cd} (right) we have plotted the relative uniform norm of the difference
of both solutions, i.e.,
\begin{equation}\label{eq:err2}
    \mathrm{err} = \max_{\substack{0.1 \leq x,y\leq 0.9}}
    | u_1(x,y) - u_2(x,y) |\Big/ \max_{\substack{0.1 \leq x,y\leq 0.9}} | u_1(x,y)|.
\end{equation}
Again we observe exponential convergence, and the reduction of the measured error is in good agreement with (even slightly better than) the rate $\rho=0.59$ expected from Theorem~\ref{T2}.

\begin{figure}[bth]
\begin{center}
\begin{minipage}[h]{\textwidth}
\hspace*{-1mm}\includegraphics[width=0.545\textwidth]{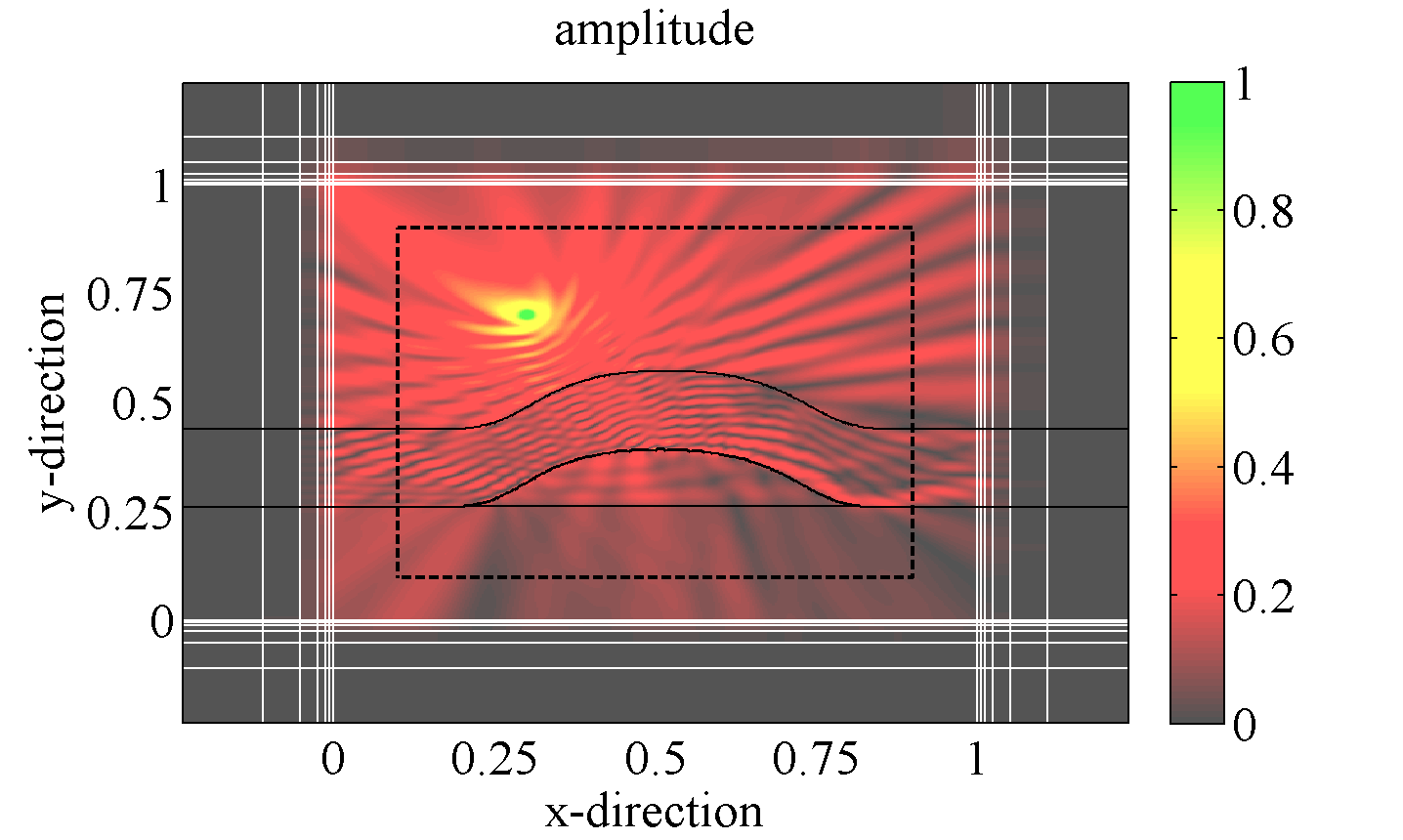}
\hspace*{-4mm}\includegraphics[width=0.545\textwidth]{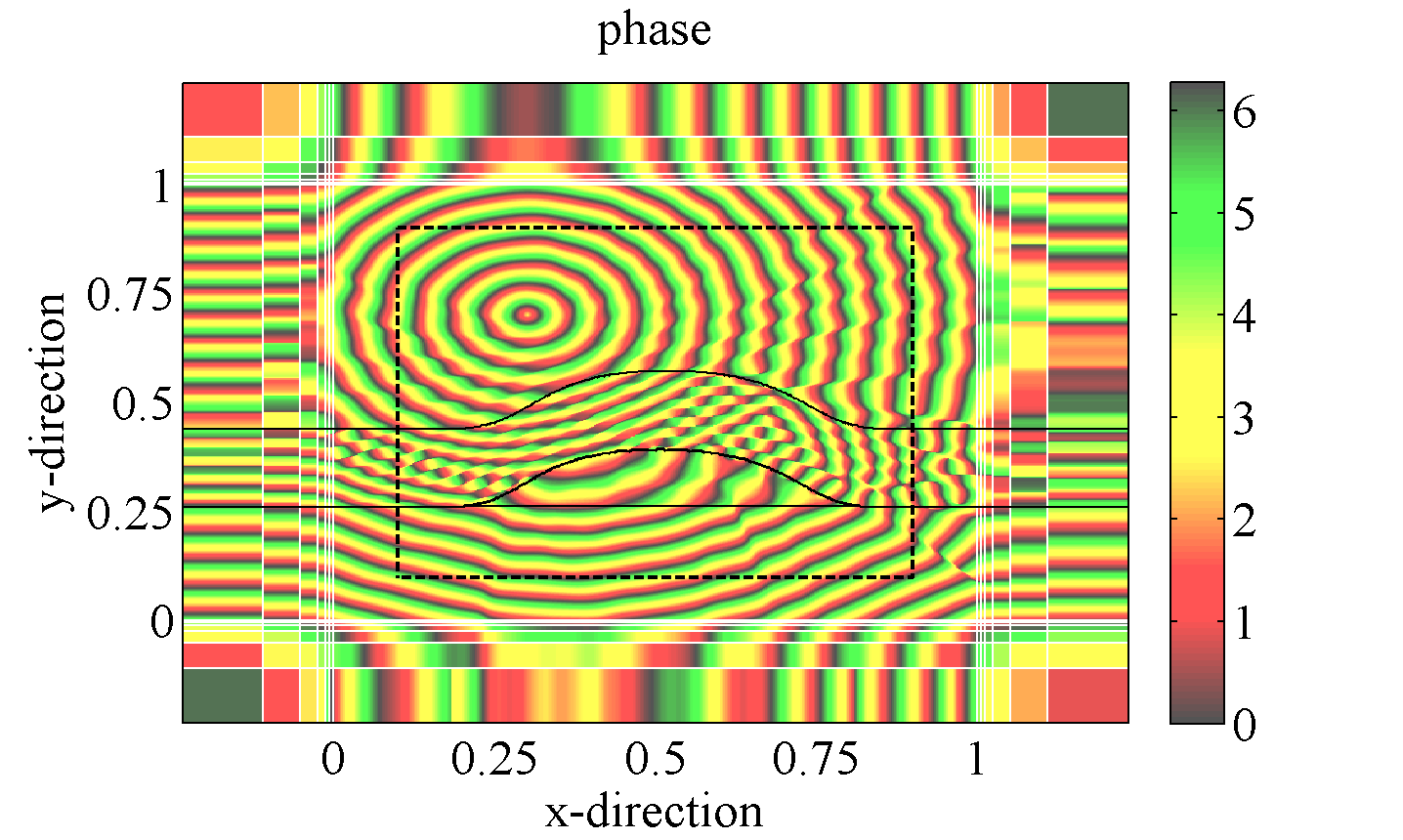}
 \end{minipage}
\end{center}
\caption{Amplitude (left) and phase (right) of the solution to the Helmholtz problem in section~\ref{sec:tensorpml} on a square domain $\Omega_1 = [0,1]^2$ appended with absorbing boundary layers at all boundary edges. In these pictures we have chosen the Zolotarev parameter $m=14$, so there are $n=7$ grid points appended to the boundaries. The step size in the interior domain is $h=1/400$ in both coordinate directions. The dashed square in the interior indicates the smaller domain $\Omega_2 = [0.1,0.9]^2$, on which we solve the same Helmholtz problem for assessing the numerical accuracy of our absorbing boundary layers.} \label{fig:example12}
\vspace*{5mm}
\begin{center}
\begin{minipage}[h]{\textwidth}
\begin{minipage}[h]{.5\textwidth}
\vspace*{-39.5mm}\includegraphics[width=1\textwidth]{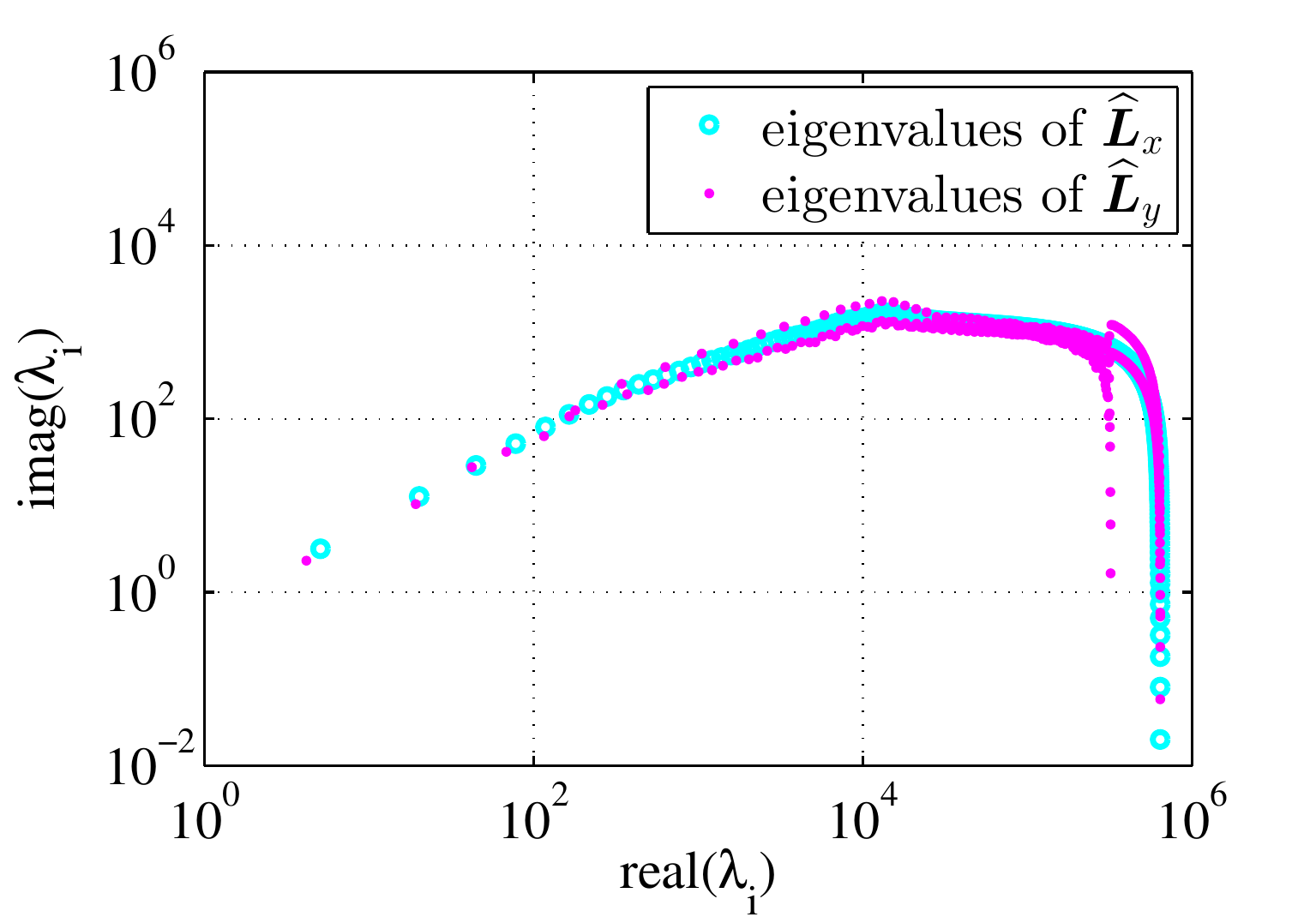}
\end{minipage}
\vspace*{0mm}\includegraphics[width=0.475\textwidth]{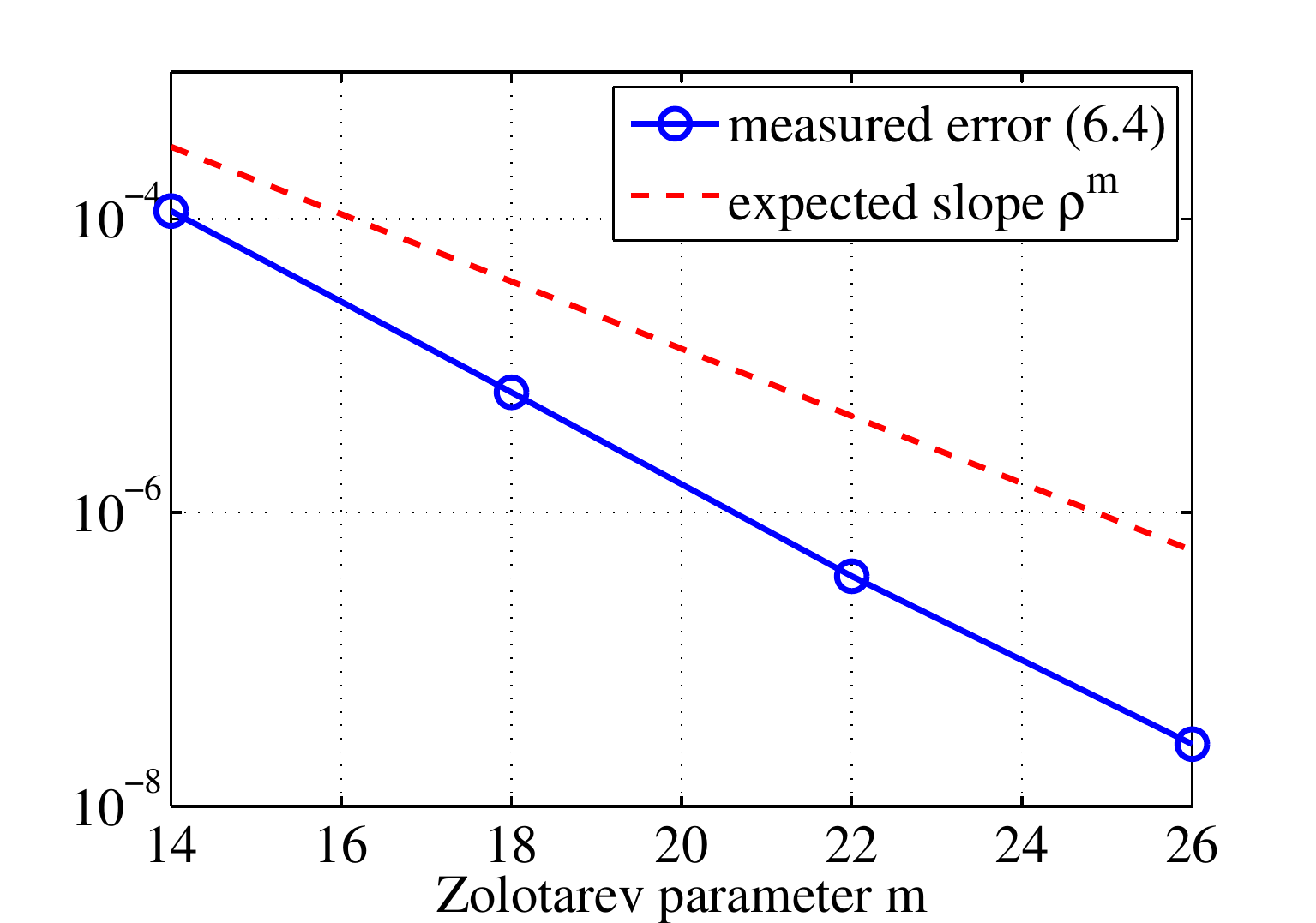}
\end{minipage}
\end{center}
\caption{Left: Eigenvalues of the matrices $\mathbf{\widehat L}_x$ and $\mathbf{\widehat L}_y$ associated with the Helmholtz problem in section~\ref{sec:tensorpml},  appended with $n=7$ grid points at the boundaries. Right: Exponential convergence of the accuracy of the absorbing boundary layers
with varying Zolotarev parameter $m\in\{14,18,\ldots,26\}$ (twice the number of grid points in each  absorbing boundary layer).
The expected convergence rate $\rho\approx 0.59$ by Theorem~\ref{T2} is indicated by the dashed line.} \label{fig:example12cd}
\end{figure}

\rev{\section{Summary, generalizations, and open problems}

We have presented a new approach for the construction of  discrete absorbing boundary layers for indefinite Helmholtz problems via complex coordinate transforms. This approach is based on the use of near-optimal relative
rational interpolants of the inverse square root (or a modification thereof) on a negative and a positive real interval. Bounds for the approximation error have been derived, and the exponential convergence of the approximants has been established theoretically and demonstrated at numerical examples. Although our focus in this paper was on absorbing boundary conditions for indefinite Helmholtz problems, it was recently understood that these conditions also constitute good approximations to Schur complements of certain PDE discretization matrices, and they became a crucial component of modern Helmholtz preconditioners, such as  Schwarz domain decomposition \cite{Gander,BAG12} and the sweeping preconditioner in \cite{Engquist}. Preliminary results have shown successful application to a
multilevel domain decomposition preconditioner, and a related Schlumberger  patent application   is pending.

\subsection{Time-domain problem} 
Classical (explicit)  finite-difference time-domain formulations  lead to PMLs that can be represented  via  grid steps $\gamma_j$ and $\widehat \gamma_j$  which are dependent on the wave number $k$ as
$ \gamma_j=\alpha_j+\frac{\beta_j}{i k}$, $\widehat \gamma_j=\widehat\alpha_j+\frac{\widehat\beta_j}{i k}$, where $\alpha_j,\beta_j,\widehat \alpha_j,\widehat\beta_j$ are  real positive parameters  \cite{Hagstrom, Ber94,Joly,ChewWeedon}. 
Our experiments suggest that the steps of our optimal PML grids always have positive real parts, and negative or zero imaginary parts; see, e.g., the grid in Figure~\ref{fig:example1c}. 
So formally, the steps of the frequency-dependent PML can be obtained as
$ \gamma_j=\Re h_j+i\Im h_j\frac{k_0}{k}$, $\widehat \gamma_j=\Re \widehat h_j+i\Im \widehat h_j\frac{k_0}{k}$,
where the steps $h_j$  and $\widehat h_j$ are obtained for a fixed wave number $k_0$. If the rational approximant  for $k=k_0$ uses  symmetric intervals of approximation, 
the corresponding grid lies on a semiaxis rotated by $-\frac{\pi}{4}$ with respect to $\bfR_+$ and  the  introduction of $k\ne k_0$ is equivalent to the rotation of the grid and the spectral measure respectively on  $\bfC\setminus \bfR_+$ and  
$\bfC\setminus \bfR_-$. Therefore such grids retain the exponential convergence for $z\in \bfR_-$. However, it is not clear if
the exponential convergence holds for nonsymmetric intervals, neither is known if this convergence holds for the approximation of the discrete impedance function from section~\ref{sec:discr}.

Alternatively, the time domain solution can be represented via  stability-correcting functions of the discretized operator with the PML obtained for a fixed wave number; see~\cite{DrRemis}. It can then be  efficiently computed  in the time domain  via Krylov subspace projection.

\subsection{Maxwell and elasticity  systems}
Important hyperbolic systems, such as {\it isotropic} Maxwell's and linear elasticity systems are usually approximated via staggered finite-difference (finite volume) schemes \cite{Yee,Virieux}. Staggered schemes for multidimensional  problems  can be constructed via tensor products of one-dimensional staggered schemes. Thus  the one-dimensional staggered grids developed in this paper can be automatically implemented in such framework, similar to what was done in   \cite{davydycheva_etal2002,Lisitsa}  for PMLs based on the single interval rational approximants. Work \cite{Lisitsa} also provides    error estimates for the propagative modes in   isotropic elasticity systems, showing that for  hyperbolic systems one may need rational  approximants on slightly larger spectral intervals compared to the scalar wave equation.  

\subsection{Adaptive grids}
The uniform approximation approach  requires bounds for the smallest/largest negative and positive eigenvalues, which can be rather loose due to the weak dependence of the convergence rate of the Zolotarev approximants on the interval ratios. The external bounds of the intervals can thus be
estimated roughly.
 Still, the numerical estimation of the internal bounds can be rather difficult, and accidentally at least one eigenvalue  may be very close to the origin, in which case even an optimal approximant may require significant order for a satisfactory accuracy.
To circumvent this problem, it would be interesting to derive a parameter-free near-optimal rational  approximant of $\mathbf{A}^{-1/2} \mathbf{b}$, which takes into account the discrete nature of the spectrum of $\mathbf{A}$ and the spectral weights of the vector $\mathbf{b}$. Promising  steps have been made by using adaptive rational Krylov algorithms \cite{GK13,Gut13,BG14} for this purpose.

\subsection{Variable coefficients}
As explained in this paper, variable PDE coefficients in the tangential direction can be straightforwardly incorporated into the PML by modifying $\mathbf A$. Moreover,  according to  preliminary findings,  our PML approach may be  generalized  for coefficients varying in the normal direction.
Let us  replace the first equation of \eqref{eq:bvp} by 
    $$\frac{\partial^2}{\partial x^2} \mathbf{u}(x)+c(x)\mathbf{u}(x) = \mathbf{A} \mathbf{u}(x),$$
with compactly supported  coefficient $c(x)\in L_\infty$. Then grid steps $h_j,\hat h_j$, $j=1,\ldots,n$ can be obtained   via rational approximation of  the ``fine grid'' finite-difference NtD map $F^h_c=u_0/\frac{u_{0} - u_{-1}}{h}$, involving the finite-difference system $$\frac{1}{h}
 \left( \frac{u_{j+1} - u_j}{h} -  \frac{u_{j} - u_{j-1}}{h} \right)
 +c(jh) u_j=z u_j, \qquad j = 0,1,\ldots,\infty$$ with boundary conditions
\eqref{gbc}. For constant $c$ such an approach  yields $F^h_c$ being the same as  $F^h$  defined in section~\ref{sec:discr}.
Our experiments with  discrete PMLs for variable coefficients exhibited exponential convergence albeit at a slower rate than for the constant coefficient problem.

\subsection{Connection to inverse problems}
Constructing PMLs can be viewed as finding  equivalent media matching the NtD maps, and this is reminiscent to what is done in inverse problems of  electrical impedance tomography (EIT). In fact, the conversion of rational approximations to finite-difference schemes (and its planar generalization) was the basis for the solution of EIT inverse problems via resistor network approximations \cite{Borcea}. 

Finally, we would like to point out that cloaking problems (which are popular in the inverse problems community) are closely related with the construction of PMLs,
because  the latter can be viewed as cloaking of the point at infinity. Cloaking problems can also  be formulated via complex coordinate transforms \cite{Kohn} and  lead to approximation problems of NtD maps. Although the involved Stieltjes impedance function $F(z)$ is typically different in these applications, techniques similar  to those presented in this paper may still be applicable.}

\section*{Acknowledgments}
We are grateful to  Paul Childs, Martin Gander, Mikhail Zaslavsky, and Hui Zhang for useful discussions. \rev{We thank the anonymous referees for helpful comments and suggestions, and}  David Bailey for making available the Fortran~90 multiprecision system \cite{Bailey}.

\appendix

\section{Zolotarev approximation and proof of Theorem~\ref{T2}}

The solution of the Zolotarev problem (\ref{inZolpr}) can be computed as
\begin{equation}\label{zolotar}Z^{(c,d)}_m(z)=\prod_{j=1}^m(z-s^{(c,d)}_j),\qquad
s^{(c,d)}_j=d\cdot\dn\left(\frac{(2m-2j+1)\Kei(\delta^\prime)}{2m},\delta^\prime\right),
\end{equation}
where
\[
\delta=c/d, \quad \delta^\prime=\sqrt{1-\delta^2},
\]
\begin{equation}\label{eq:ellipke}\nonumber
\Kei(\delta) = \int_0^1 \frac{1}{\sqrt{(1-t^2)(1-\delta^2 t^2)}}\,\mathrm{d}t
\end{equation}
is the
complete elliptic integral of the first kind%
\footnote{The definition of $K(\delta)$ is not
consistent in the literature. We stick to the definition used in
\cite[Ch.~VI]{Neh75}. In \textsc{Matlab} one would type \texttt{ellipke(delta{\textasciicircum}2)}
to obtain the value $K(\delta)$.}
and where the Jacobian elliptic function $\dn$ is defined via another
such function, $\sn$, by the relations
\[
\dn(u,\kappa)=\sqrt{1-\kappa^2\sn(u,\kappa)}, \qquad
\xi = \sn(u;\kappa), \qquad
u = \int_0^{\xi}\frac{dt}{\sqrt{(1-t^2)(1-\kappa^2t^2)}}.
\]

In order to prove near-optimality results, we first
 need to study the quantity $E^{(c,d)}_m$ in (\ref{inZolpr}) carefully.
Evidently, $E^{(c,d)}_m<1$.
Upper and lower bounds for (\ref{inZolpr}) were given in \cite{ML05} as
\be \label{Ecdest}
\frac{2\exp\left(-\frac{\pi \Kei(\mu^\prime)}{4\Kei(\mu)}m\right)}
{1+\left[\exp\left(-\frac{\pi \Kei(\mu^\prime)}{4\Kei(\mu)}m\right)\right]^2}
\le
E^{(c,d)}_m \le 2\exp\left(-\frac{\pi \Kei(\mu^\prime)}{4\Kei(\mu)}m\right)
\ee
with
\[
\mu=\left(\frac{1-\sqrt\delta}{1+\sqrt\delta}\right)^2 \quad \mbox{and}
\quad \mu^\prime=\sqrt{1-\mu^2}.
\]
Hence the Cauchy--Hadamard convergence rate can be computed as
\be\label{rho12} \rho^{(\delta)}= \exp\left(-\frac{\pi\Kei(\mu^\prime)}{4\Kei(\mu)}\right).\ee
Recalling the equalities (\ref{eq:rho12})
and (\ref{defrho}), let us define the sets
\[
\tK=F(K), \quad \tK_1=F(K_1), \quad \tK_2=F(K_2),
\]
and consider the following auxiliary problem: find a (complex) monic polynomial $H_m$
of degree~$m$ being the minimizer of
\be \label{probl1}
\min_{H\in\mathcal{P}_m}
\max_{s\in\tK} \left|\frac{H(s)}{H(-s)}\right|.
\ee
We now construct an approximate solution of this problem and show that
the approximate solution gives the maximum in (\ref{probl1}) which
yields the best possible functional value up to a moderate multiplier.

 Accounting, as it was done in
\cite[Section~2]{DGH}, that
\[
\left|\frac{Z^{(\sqrt{-b_1},\sqrt{-a_1})}_{m_1}(-is)}
{Z^{(\sqrt{-b_1},\sqrt{-a_1})}_{m_1}(is)}\right|=1
\qquad \mbox{if} \quad
s\in\tK_2
\]
and
\[
\left|\frac{Z^{(\sqrt{a_2},\sqrt{b_2})}_{m_2}(s)}
{Z^{(\sqrt{a_2},\sqrt{b_2})}_{m_2}(-s)}\right|=1
\qquad \mbox{if} \quad
s\in\tK_1
\]
because these polynomials have real coefficients, the polynomial $H_m$ defined in (\ref{defh})
satisfies
\be \label{fsegm}
\max_{s\in \tK_1} \left|\frac{H_m(s)}{H_m(-s)}\right|
= E^{(\sqrt{a_2},\sqrt{b_2})}_{m_2}
\ee
and
\be \label{ssegm}
\max_{s\in\tK_2} \left|\frac{H_m(s)}{H_m(-s)}\right|
= E^{(\sqrt{-b_1},\sqrt{-a_1})}_{m_1}.
\ee

\begin{lemma} \label{T1}
The polynomial $H_m$ defined in (\ref{defh}) satisfies the inequality
\be \label{T1res1}
\max_{s\in\tK} \left|\frac{H_m(s)}{H_m(-s)}\right|
\le 2\max\left\{\rho_1^{-1/2},\rho_2^{-1/2}\right\}\rho^m
\ee
with the numbers $\rho_1,\rho_2$ and $\rho$ defined in \eqref{eq:rho12} and \eqref{defrho},
provided that $m_1,
m_2$ are chosen according to (\ref{m12}).

On the other hand, for any complex polynomial $H\in\mathcal{P}_m$ we have
\be \label{T1res2}
\max_{s\in\tK} \left|\frac{H(s)}{H(-s)}\right|
\ge \rho^m.
\ee
\end{lemma}

\begin{proof}
Let $H_m$ be defined as in (\ref{defh}) and conditions~(\ref{m12}) be satisfied. Accounting for
(\ref{fsegm}), (\ref{ssegm}) and (\ref{Ecdest}), we obtain
\[
\max_{s\in\tK} \left|\frac{H_m(s)}{H_m(-s)}\right|
\le 2\max\left\{\rho_1^{m_1},\rho_2^{m_2}\right\}
=2\rho^m\max\left\{\rho_1^\theta,\, \rho_2^{-\theta}\right\},
\]
which gives assertion (\ref{T1res1}).

To prove assertion~(\ref{T1res2}), we consider the third
Zolotarev problem in the complex plane for the condenser
$\big(\tK,-\tK\big)$ (see \cite{Gonchar69},
\cite[\S~8.7]{Walsh} or \cite[\S~VIII.3]{SaffTotik}). Due to the symmetry of the condenser, the two measures forming the (unique) equilibrium pair for
$\big(\tK,-\tK\big)$ are symmetric to each other in the evident
sense. Thus, one can choose an (in the Cauchy--Hadamard sense) optimal
 sequence of type $(m,m)$ rational functions of the
form $H(s)/H(-s)$, $\deg(H)=m\geq 1$, such that the
roots $s_j$ ($1\le j \le m$) of each polynomial $H$ belong to $\tK$. Define
\beas
H^{(1)}(s)=\prod_{\substack{1\le j\le m\\  s_j\in\tK_1}}(s-s_j), \qquad \deg( H^{(1)} )=m_1,\\[2mm]
H^{(2)}(s)=\prod_{\substack{1\le j\le m\\ s_j\in\tK_2}}(s-s_j), \qquad \deg( H^{(2)} )=m_2,\\
m_1+m_2=m.\ \,
\eeas
By virtue of (\ref{Ecdest}) and the location of the roots we have
\[
\max_{s\in\tK_1}\left|\frac{H(s)}{H(-s)}\right|
=\max_{s\in\tK_1}\left|\frac{H^{(1)}(s)}{H^{(1)}(-s)}\right|
\ge\frac{2\rho_1^{m_1}}{1+\rho_1^{2m_1}}
\]
and
\[
\max_{s\in\tK_2}\left|\frac{H(s)}{H(-s)}\right|
=\max_{s\in\tK_2}\left|\frac{H^{(2)}(s)}{H^{(2)}(-s)}\right|
\ge\frac{2\rho_2^{m_2}}{1+\rho_2^{2m_2}}\, ,
\]
whence
\beas
\max_{s\in\tK}\left|\frac{H(s)}{H(-s)}\right|\cdot
\max_{s\in-\tK}\left|\left[\frac{H(s)}{H(-s)}\right]^{-1}\right|
=\max_{s\in\tK}\left|\frac{H(s)}{H(-s)}\right|^2
\ge \max\left\{\frac{2\rho_1^{m_1}}{1+\rho_1^{2m_1}},
\frac{2\rho_2^{m_2}}{1+\rho_2^{2m_2}}\right\}^2 \\
\ge \max\left\{\rho_1^{2m_1},\rho_2^{2m_2}\right\}.
\eeas
Since the quantity $\max\left\{\rho_1^{2x_1},\rho_2^{2x_2}\right\}$ under
the conditions $x_1\ge0$, \ $x_2\ge0$, \ $x_1+x_2=m$ is minimal at
\[
x_1=m\cdot \frac{\log\rho_2}{\log\rho_1+\log\rho_2},
\qquad
x_2=m\cdot \frac{\log\rho_1}{\log\rho_1+\log\rho_2},
\]
we obtain
\[
\max_{s\in\tK}\left|\frac{H(s)}{H(-s)}\right|\cdot
\max_{s\in-\tK}\left|\left[\frac{H(s)}{H(-s)}\right]^{-1}\right|
\ge \rho^{2m}
\qquad\mbox{as}\quad m\to\infty,
\]
so
\[
\liminf_{m\to\infty}
\left(
\max_{s\in\tK}\left|\frac{H(s)}{H(-s)}\right|\cdot
\max_{s\in-\tK}\left|\left[\frac{H(s)}{H(-s)}\right]^{-1}\right|
\right)^{1/m}
\ge \rho^2.
\]
It follows in view of \cite[Theorem~1, Formula~(12)]{Gonchar69} that the
logarithmic capacity of our condenser satisfies
\[
\exp\left(-1/\capl\big(\tK,-\tK\big)\right)\ge\rho^2.
\]
Moreover, \cite[Theorem~1, Formula~(11)]{Gonchar69} yields for all $H,G\in\mathcal{P}_m$
\[
\max_{s\in\tK}\left|\frac{H(s)}{G(s)}\right|\cdot
\max_{s\in-\tK}\left|\frac{G(s)}{H(s)}\right|
\ge \rho^{2m},
\]
from which (\ref{T1res2}) follows.\hfill
\end{proof}

We are now prepared to conclude the proof of Theorem~\ref{T2}.

\begin{proof}
To establish (\ref{T2res1}), it suffices to note that
\[ 
\max_{z\in K} \left|\frac{R_n(z)}{F(z)}-1\right|
=\max_{s\in\tK} \left|\frac{s\, P_{n-1}(s^2)}{Q_n(s^2)}-1\right|,
\]
and to apply (\ref{T1res1}) from Lemma~\ref{T1}, condition
(\ref{T2cond}), and a consequence of (\ref{apprerr}) for finding
\[
\left|\frac{sP_{n-1}(s^2)}{Q_n(s^2)}-1\right|
=\frac{2\left|\frac{H_m(s)}{H_m(-s)}\right|}
{\left|1+\frac{H_m(s)}{H_m(-s)}\right|}\, .
\]

To justify (\ref{T2res2}), set $z=s^2$ and   assume that for
some pair $(P,Q)$ and $R=P/Q$  we have the
inequality
\[
\max_{z\in K} \left|\frac{R(z)}{F(z)}-1\right|
< \frac{2\rho^m}{1+\rho^m}.
\]
Define $H$ by means of (\ref{defpq}) and rewrite the equality
(\ref{apprerr}) in the form
\[
\frac{H(s)}{H(-s)}=- \frac{\frac{R(z)}{F(z)}-1}
{\left(\frac{R(z)}{F(z)}-1\right)+2}.
\]
We readily derive
\[
\max_{s\in\tK}\left|\frac{H(s)}{H(-s)}\right|
< \frac{\frac{2\rho^m}{1+\rho^m}}{2-\frac{2\rho^m}{1+\rho^m}}
=\rho^m,
\]
which contradicts (\ref{T1res2}) and thereby proves the assertion (\ref{T2res2}).\hfill
\end{proof}

\end{document}